\documentclass[reqno,11pt]{amsart}

\usepackage{amsthm, amsmath, amssymb, stmaryrd}
\usepackage[utf8]{inputenc}
\usepackage[T1]{fontenc}

\usepackage[hyphenbreaks]{breakurl}
\usepackage[hyphens]{url}
\usepackage{systeme}
\usepackage[shortlabels]{enumitem}
\usepackage[hidelinks]{hyperref}
\usepackage{microtype}

\usepackage{bm}
\usepackage[margin=1in]{geometry}

\usepackage[textsize=scriptsize,backgroundcolor=orange!5]{todonotes}

\usepackage[noabbrev,capitalize,sort]{cleveref}
\crefname{equation}{}{}
\crefname{observation}{Observation}{Observations}
\crefname{claim}{Claim}{Claims}
\numberwithin{equation}{section}

\usepackage{mathtools}
\usepackage{bbm}

\theoremstyle{plain}
\newtheorem{theorem}{Theorem}[section]
\newtheorem{corollary}[theorem]{Corollary}
\newtheorem{proposition}[theorem]{Proposition}
\newtheorem{lemma}[theorem]{Lemma}
\newtheorem{claim}[theorem]{Claim}

\newtheorem{observation}[theorem]{Observation}

\theoremstyle{definition}

\newtheorem{remark}[theorem]{Remark}
\newtheorem{definition}[theorem]{Definition}

\newtheorem{def/prop}[theorem]{Definition/Proposition}

\newcommand*{\PP}{\mathbb{P}}  
\newcommand*{\veps}{\varepsilon}                                 % Nice-looking epsilon
\newcommand*{\EE}{\mathbb{E}}                                     % Expectation

\newcommand{\abs}[1]{\left\lvert#1\right\rvert}
\newcommand*{\FF}{\mathbb{F}}  
\newcommand*{\parity}{\textup{par}}
\newcommand*{\cJ}{\mathcal{J}}

\newcommand{\cP}{\mathcal P}
\newcommand{\cE}{\mathcal E}

\newcommand{\NN}{\mathbb N}
\newcommand{\cI}{\mathcal{I}}
\newcommand{\cQ}{\mathcal{Q}}
\newcommand{\ptrU}{\mathfrak{p}}
\newcommand{\ptrW}{\mathfrak{q}}
\renewcommand{\P}{\mathbb{P}}

\newenvironment{proof*}[1][\proofname]{
  
  \begin{proof}[#1]}{\end{proof}}

\newcommand{\lpr}[1]{\left(#1\right)}
\newcommand{\lbr}[1]{\left[#1\right]}
\newcommand{\lcr}[1]{\left\{#1\right\}}

\newcommand{\floor}[1]{\lfloor#1\rfloor}

\newcommand{\labs}[1]{\left|#1\right|}
\newcommand{\bin}{\{0,1\}}
\newcommand{\eps}{\varepsilon}
\newcommand{\supp}{\textup{supp}}
\newcommand{\lmid}{\,\middle\vert\,}

\newcommand\coin{
\mathrel{\ooalign{\hss$\bigcirc$\hss\cr\kern0.7ex\hbox{\scalebox{0.8}{$\$$}}}}}

\makeatletter
\newcommand\thankssymb[1]{\textsuperscript{\@fnsymbol{#1}}}
\makeatother

\title{Even-degeneracy of a random graph}
\author[Ting-Wei Chao]{Ting-Wei Chao\thankssymb{1}}
\author[Dingding Dong]{Dingding Dong\thankssymb{2}}
\author[Zixuan Xu]{Zixuan Xu\thankssymb{1}}

\thanks{\thankssymb{1}Department of Mathematics, Massachusetts Institute of Technology, Cambridge, MA, USA. Email: {\tt \{twchao, zixuanxu\}@mit.edu}}

\thanks{\thankssymb{2}Department of Mathematics, Harvard University, Cambridge, MA, USA. Email: {\tt ddong@math.harvard.edu}}

\begin{document}

\begin{abstract}
A graph is even-degenerate if one can iteratively remove a vertex of even degree at each step until at most one edge remains. Recently, Janzer and Yip showed that the Erd\H{o}s--R\'{e}nyi random graph $G(n,1/2)$ is even-degenerate with high probability, and asked whether an analogous result holds for any general $G(n,p)$. In this paper, we answer this question for any constant $p\in (0,1)$ in affirmation by proving that $G(n,p)$ is even-degenerate with high probability.
\end{abstract}

\maketitle

\section{Introduction}

For a graph $G$ on $n$ vertices, we say that $G$ is \emph{even-degenerate} if there exists an ordering $v_1,\dots,v_n$ of its vertices such that for every $1\leq i\leq n-2$, $v_i$ has an even number of neighbors in the set $\{v_{i+1},\dots,v_n\}$. In other words, $G$ is even-degenerate if one can sequentially remove its vertices such that every step removes an even number of edges, until at most two vertices remain.

Note that the parity of the number of remaining edges in the graph never changes during this removal. This means that when the edge number $e(G)$ of $G$ is even, one can legally remove every vertex of the graph; but when $e(G)$ is odd, there must be a leftover edge that is not legally removable. Therefore the property of being even-degenerate does not require $e(G)$ to be even. 

In the case of $e(G)$ being even, an even-degenerate graph is in fact \emph{even-decomposable}, a graph property previously studied by Versteegen \cite{Ver25} and Janzer and Yip \cite{JY24} due to motivations coming from the study of graph codes initiated by Alon, Gujgiczer, K\"orner, Milojevi\'c, and Simonyi~\cite{AlonGKMS23,Alon24}. A graph $G$ is \emph{even-decomposable} if there exists a sequence of vertex subsets $V(G)=V_1\supseteq V_2\supseteq \dots \supseteq V_k=\varnothing$ such that for every $1\leq i\leq k-1$, $G[V_i]$ has an even number of edges and $V_i\setminus V_{i+1}$ is an independent set in $G$.

Motivated by applications in estimating the sizes of linear graph codes, Versteegen showed that for $G\sim G(n,1/2)$, the probability that $G$ is not even-decomposable is at most $e^{-\Omega(\sqrt{\log n})}$. The bounds are later improved by Janzer and Yip \cite{JY24} to $e^{-\Omega(n^2)}$. Janzer and Yip also determined the threshold $p(n)$ for which $G\sim G(n,p(n))$ is very likely to be even-decomposable. In particular, they showed the following.

\begin{theorem}[{\cite[Theorem 1.10 and Proposition 1.11]{JY24}}]
    There exist constants $c_1,c_2 > 0$ such that for any $p\le 1-\frac{c_1}{n}$ and $G\sim G(n,p)$,
    \[\PP[G \text{ is not even-decomposable}\mid e(G)\text{ is even}\,] = o(1).\]
    For any $p \ge 1-\frac{c_2}{n}$ and $G\sim G(n,p)$,
    \[\PP[G \text{ is even-decomposable}\,] = o(1).\]
\end{theorem}

In the same work, motivated by a question of Alon, Janzer and Yip initiated the study of the probability of a random graph being even-degenerate, as even-degeneracy is a stronger and more natural property than even-decomposability.
They were able to determine the correct asymptotics for the probability of a random graph $G\sim G(n, 1/2)$ being even degenerate.

\begin{theorem}[{\cite[Theorem 1.13]{JY24}}]\label{thm:JY24}
    If $G\sim G(n,1/2)$, then $\PP[\text{$G$ is not even-degenerate}] = e^{-\Theta(n)}$.
\end{theorem}

The lower bound side in \cref{thm:JY24} is  straightforward, because with probability $\geq (1/2)^{2n}$ either every vertex in $G\sim G(n,1/2)$ has odd degree (when $n$ is even), or $G\sim G(n,1/2)$ has all odd-degree vertices except for an isolated vertex (when $n$ is odd). The upper bound side is more involved, and no analogue of \cref{thm:JY24} is known for any other values of $p$ aside from $p = 1/2$. This is because the proof of \cref{thm:JY24} employs an inductive strategy, which critically uses the fact that $p$ is exactly $1/2$. More specifically, the proof relies on the observation that given every non-empty vertex subset $S\subseteq V$, the random graph $G[V\setminus S]$ conditioned on the parities of $\deg_G(v)$ for all vertices $v\in V$ is identically distributed as $G(V\setminus S, 1/2)$. Thus, the proof strategy in \cite{JY24} cannot be directly generalized to $G(n, p)$ for any $p \neq 1/2$.

In our work, we make substantial progress towards proving an analogous result to \cref{thm:JY24} for any constant $p\in (0,1)$.

\begin{theorem}\label{thm:main}
    For any constant $p \in (0,1)$ and any $\alpha > 0$,
    \[\PP_{G\sim G(n,p)}[G \text{ is not even-degenerate}] = e^{-\Omega(n^{1/2 - \alpha})}.\]
\end{theorem}

In particular, \cref{thm:main} shows that for any constant $p\in (0,1)$, even-degenerate graphs are always dominant in $G(n,p)$. Using the same lower bound construction as in \cref{thm:JY24}, one can show a lower bound of $e^{-O(n)}$ on the above probability. Thus determining the correct asymptotics remains an interesting open question.

\medskip 

\paragraph{\textbf{Paper Organization}}

In \cref{sec:proof-strategy}, we highlight the main obstruction to generalizing the proof strategy for the case where $p = 1/2$ and give a high level overview of our proof. In \cref{sec:prelim}, we introduce some useful definitions and probabilistic lemmas. In \cref{sec:uw-removal,sec:algorithm,sec:upperbound-f}, we present the proof of \cref{thm:main}. We start by defining our removal procedure and analyzing its success probability in \cref{subsec:lr-removal}. Then in \cref{sec:algorithm}, we show how to obtain our main recurrence (\cref{thm:recurrence}) based on the removal procedure, where we delay the proof of a technical lemma to \cref{subsec:comb-lemma,subsec:comb-lemma-2}. Finally, we analyze the recurrence in \cref{thm:recurrence} and complete the proof of our main result \cref{thm:main} in \cref{sec:upperbound-f}.

\noindent \textbf{Notations}. For a positive integer $n$, we use $[n]$ to denote the set $\{1,\dots, n\}$. We abuse notation and use the convention that $a\pmod 2$ takes value in $\{0,1\}$ but $a\pmod s$ for $s > 2$ takes value in $\{1,\dots, s\}$ for convenience. When we write $b= a\pmod{s}$ it means that $b$ is the same as $a\pmod{s}$, but $b\equiv a\pmod{s}$ only means that $a$ and $b$ have the same remainder modulo $s$. For a graph $G = (V,E)$, we use $N_G(v)$ for $v\in V$ to denote the set of neighbors of $v$ in $G$. Similarly, for a subset $A\subseteq V$, we use $N_G(A)$ to denote the neighbors of $A$ in $G$.

\section{Proof strategy}\label{sec:proof-strategy}

To see why it is nontrivial to generalize the proof strategy in \cite{JY24} beyond $p\ne 1/2$, let us first briefly discuss how Janzer and Yip obtained the preliminary  upper bound $$\PP_{G\sim  G(n,1/2)}[G\text{ is not even-degenerate}]=o(1)$$
which is an intermediate step in their proof of \cref{thm:JY24}.
For $p = 1/2$, we have the nice property that given $G\sim G(n,1/2)$ conditioned on the parities of $\deg_G(x)$ for all $x\in V$, the graph $G\setminus \{v\}$ for any $v\in V$ is identically distributed as $G(n-1,1/2)$. So if one defines $f(n)$ as the probability of $G\sim G(n,1/2)$ being not even-degenerate, it is easy to obtain a recursive relation between $f(n)$ and $f(n-1)$ by exposing the parity of the degree of every vertex in $G$ and analyzing the correlation between removals of different vertices of even degree.

For other constants $p\in(0,1)$ (especially for $p$ close to $0$ or $1$), the above property is far from being true. If we expose the degree parities of every vertex in $G\sim G(n,p)$, then once we remove a vertex $v$ from $G$, the distribution of $G\setminus\{v\}$ is different from $G(n-1,p)$.

Despite the fact that such a recursion does not hold, the random graph $G(n,p)$ with $0<p<1$ still has some nice properties that we can utilize. For example, for each vertex $v\in V$, the probability that $v$ has even degree is
\[
\sum_{\substack{0\leq i\leq n\\i\text{ is even}}}\binom{n}{i}p^i(1-p)^{n-i}=\frac{1}{2}\pm o(1).
\]
This guarantees that we can almost always remove one vertex from $G$ when $n$ is large enough.

The hope is that as long as the remaining graph has enough randomness, we can find a vertex with even degree and remove it. To ensure this, we need to develop a careful removal procedure, which motivates the ``$(U,W)$-removal'' procedure defined in our proof.
We now move on to give a simplified overview of the $(U,W)$-removal; the full details can be found in \cref{def:lrremoval}.

\subsection{$(U,W)$-removal}

For $G\sim G(n,p)$, the $(U,W)$-removal on $G$ roughly proceeds as follows: we fix a partition of the vertex set $V(G)=B\sqcup C$ where the sizes of $B$ and $C$ are almost the same. We assign a fixed ordering $B=\{b_1,\dots, b_{|B|}\}$ to the vertices in $B$, and we partition $C$ into $s+1 = \Theta(n^{1/2-\alpha})$ sets $C_1,\dots, C_s,C_{\#}$ of equal sizes where $\alpha>0$ is small. Our goal is to remove all the vertices in $B$ in the assigned ordering, although we might also remove some vertices in $C$ along the way.

Observe that when we try to remove $b_i\in B$, as long as, say, a constant fraction of the vertices in $C$ are still present,  $b_i$ has even degree in the current graph with probability $1/2\pm o(1)$. If the degree of $b_i$ in the current graph is even, we can just remove $b_i$; if the degree is odd, we will try to find a vertex $u\in C$ so that we can remove both. Namely, we wish to find a vertex $u\in C$ such that
\begin{enumerate}[label=(\arabic*)]
    \item $b_iu$ is an edge,
    \item $u$ has even degree in the current graph. 
\end{enumerate}
After finding such $u$, we can first remove $u$ and then remove $b_i$. 

For every $b_i\in B$, if we happen to require such a vertex $u\in C$, we will only look for one in one part $C_j$ of the partition, and we loop through the parts $C_1,\dots, C_s$ every time we need such a vertex. The advantage of this is that every time we get back to looking for a vertex in $C_j$, we have already removed at least $s$ new vertices from $G$ so that the parities (2) we query are again close to being uniformly random.

Once we succeed in the $(U,W)$-removal (i.e., have successfully removed all vertices in $B$), we will recurse and apply the procedure on the remaining graph. Note that after the removal, it is likely that roughly $\abs{B}/2$ vertices in $C$ are removed. Thus, if $\abs{B}=\abs{C}=n/2$, then there will be roughly $n/4$ vertices left. Let $f(n)$ be the probability of $G(n,p)$ not being even-degenerate. This suggests that $f(n)\leq f(n/4)+o(1)$, assuming that the $(U,W)$-removal succeeds with high probability.

If we apply $(U,W)$-removal only sequentially, we may end up with a constant number of vertices remaining, so we cannot get an upper bound on $f(n)$ better than $f(K)$ for some constant $K$. To solve this problem, we also consider the same procedure where the roles of $B$ and $C$ are interchanged. If at least one of the two removal processes succeeds, then we can recurse on the remaining graph. Note that in the two removal processes, the sets of remaining vertices are disjoint, and thus we may expect that the remaining subgraphs are close to being independent. This suggests that $f(n)\leq f(n/4)^2+o(1)$ which is more amenable to proving a bound of the form $f(n)=o(1)$.

However, it is clear that after the removal procedure, what remains is not exactly a $p$-random graph. What we are left with is a subgraph with some partial information on the parities of certain edge subsets. We will show that these partial information do not affect the further removals too much.
The main technical difficulties of our proof arise from rigorously analyzing the information we expose at each step and showing that the sets we query have almost uniform parities.

We define the quantity we want to bound recursively as follows.

\begin{definition}[$\alpha$-partially revealed $p$-graph]\label{def:prg}
    For $0<\alpha < 1/2$ and $p\in (0,1)$, we say that a random graph $G = (V,E)$ is an \emph{$\alpha$-partially revealed $p$-graph} if there exists a set $A\subseteq V$ of size at most $n^{1-2\alpha}$, such that $G$ follows the distribution $G(n,p)$ on vertex set $V$ conditioned on the following information:
    \begin{enumerate}
        \item the subgraph $G[A]$,
        \item the parities of $\deg_G(a)$ for all $a\in A$, and
        \item the parity of $|E|$.
    \end{enumerate}
    We call $A$ the \emph{revealed part} of $G$.
\end{definition}

\begin{definition}
    Let $f_\alpha(n)$ be the maximum of the probability that $G$ is not even-degenerate over all $\alpha$-partially revealed $p$-graphs $G$ with $n$ vertices.
\end{definition} 

We will show the following recursive upper bound on $f_\alpha(n)$. Solving the recurrence in \cref{thm:recurrence} directly gives \cref{thm:main} and we will deduce \cref{thm:main} in \cref{sec:upperbound-f}.
\begin{theorem}\label{thm:recurrence}
    Let $0 < \alpha < 1/2$ and $p\in (0,1)$ be constants. There exists a positive integer $N_0$ depending on $p$ and $\alpha$ such that the following holds. For any $n\geq N_0$, we have
    \[1-f_\alpha(n)\geq (1-e^{-\Theta(n^{1/2-\alpha})})  \lpr{1-\max_{n'\in [n/4-o(n),n/4+o(n)]}f_\alpha(n')^2}.\]
\end{theorem}

\section{Preliminaries}\label{sec:prelim}

In this section, we introduce some definitions that will be useful in the following proofs. We first recall the definition of the support of a random variable.

\begin{definition}[Support]
Let $X\in \Omega$ be a random variable. The \emph{support of $X$} is defined as $\supp(X):= \{x\in \Omega \mid \PP[X = x]> 0\}$.
Similarly, for an event $\cE$, we define $\supp(X\mid \cE):= \{x\in \Omega \mid \PP[X = x\mid \cE]> 0\}$.
\end{definition}

Now we introduce the notion of $\eps$-closeness to describe two random variables that are close to each other in distribution pointwise.

\begin{definition}[$\eps$-close]\label{def:eps-close}
    Let $X$ and $X'$ be two random variables and $\cE,\cE'$ be two events. We say that $X$ is \emph{$\eps$-close} to $X'$ if $\supp(X) = \supp(X')$ and for every $x\in \supp(X)$, we have
    \[1-\eps \le \frac{\PP[X' = x]}{\PP[X = x]}\le (1-\eps)^{-1}.\]
    Similarly, we say that $(X\mid\cE)$ is \emph{$\eps$-close} to $(X'\mid\cE')$ if $\supp(X\mid\cE) = \supp(X'\mid\cE')$ and for every $x\in \supp(X\mid\cE)$, we have
    \[1-\eps \le \frac{\PP[X' = x\mid\cE']}{\PP[X = x\mid\cE]}\le (1-\eps)^{-1}.\]
If $Y$ and $Y'$ are two other random variables, we say that $(X\mid Y)$ is \emph{$\eps$-close} to $(X'\mid Y')$ if $(X\mid Y=y)$ is $\eps$-close to $(X'\mid Y'=y')$ for all $(y,y')\in \supp(Y,Y')$. We write $(X\mid Y)=(X'\mid Y')$ if $(X\mid Y)$ is $0$-close to $(X'\mid Y')$.
\end{definition}

For convenience, we say that a random variable  $X$ is $\eps$-affected by an event or another random variable if conditioning on the latter only changes the distribution of $X$ by a multiplicative $(1 - \eps)$-factor. 

\begin{definition}[$\eps$-affected]
    Let $X$ be a random variable and $\cE$ be an event. We say that $X$ is \emph{$\veps$-affected} by $\cE$ if $(X\mid \cE)$ is $\veps$-close to $X$. Similarly, for a random variable $Y$ , we say that $X$ is \emph{$\veps$-affected} by $Y$ if $(X\mid Y)$ is $\veps$-close to $X$.
\end{definition}

We now define the notion of $\eps$-uniformity for a random variable over $\bin^r$.

\begin{definition}[$\eps$-uniform]\label{def:eps-uniform}
     Let $\vec{X} = (X_1,\dots, X_r)$ be a random variable over $\bin^r$. Fix $\eps > 0$, then we say that $\vec{X}$ is \emph{$\eps$-uniform} if for all $\vec{s} \in \bin^r$,
    \[\PP[\vec{X} = \vec{s}] \in \left[(1-\veps)\frac{1}{2^r}, \, (1-\eps)^{-1}\frac{1}{2^r}\right].\]
    For a random variable $Y$, we say that $(\vec X\mid Y)$ is $\eps$-uniform if $(\vec X\mid Y=y)$ is $\eps$-uniform for all $y\in\supp(Y)$.
\end{definition}

\begin{definition}[fix-parity uniform]
    Let $\vec{X} = (X_1,\dots, X_r)$ be a random variable over $\bin^r$. We say that $\vec{X}$ is \emph{fix-parity uniform} if there exists $s\in\{0,1\}$ such that
    \begin{align*}
        \PP[\vec{X} = (s_1,\dots,s_r)] =\begin{cases}
            \frac{1}{2^{r-1}} &\text{ if $s_1+\dots+s_r\equiv s\pmod{2}$,}\\
        0&\text{ if $s_1+\dots+s_r\equiv 1-s \pmod{2}$.}
        \end{cases}
    \end{align*}
\end{definition}

\begin{definition}[fix-parity $\eps$-uniform]Let $\vec{X} = (X_1,\dots, X_r)$ be a random variable over $\bin^r$.
    Fix $\eps > 0$. We say that $\vec{X}$ is \emph{fix-parity $\eps$-uniform} if $\vec{X}$ is $\veps$-close to a fix-parity uniform distribution  over $\bin^r$.
Similarly, if $Y$ is another random variable, we say that $(\vec X\mid Y)$ is fix-parity $\eps$-uniform if $(\vec X\mid Y=y)$ is fix-parity $\eps$-uniform for all $y\in\supp(Y)$.
\end{definition}

Now we introduce some useful probabilistic lemmas about $\{0,1\}$-valued random variables.

\begin{lemma}\label{lem:parity-close-to-uniform}
Let $p\in (0,1)$, $\eta\in \NN$ and let $X_1,\dots, X_{\eta}$ be independent Bernoulli random variables with probability $p$. Let $p^* = \min\{p, 1-p\}\in (0,1/2]$. Let $Y = \sum_{i=1}^\eta X_i\pmod 2$ be the parity of the sum of $X_1,\dots, X_{\eta}$. Then $Y$ is $e^{-2\eta p^*}$-uniform. 
\end{lemma}

\begin{proof}
    Since $X_1,\dots, X_{\eta}$ are independent random variables, we have
    \[
    \abs{\EE[(-1)^{Y}]}=\left|\EE\lbr{(-1)^{\sum_{i=1}^\eta X_i}}\right|=\left|\prod_{i=1}^\eta\EE\lbr{(-1)^{X_i}}\right| =(1-2p^*)^\eta \le e^{-2 \eta p^*}.
    \]
    Since $Y$ takes value in $\{0,1\}$, we have
    \[\PP[Y = 0], \, \PP[Y = 1] \in \lbr{\frac{1}{2}\lpr{1-e^{-2 \eta p^*}},\,  \frac{1}{2}\lpr{1+e^{-2\eta p^*}}}.\] 
    Since $1+e^{-2\eta p^*}\leq (1-e^{-2 \eta p^*})^{-1}$, we can conclude that $Y$ is $e^{-2 \eta p^*}$-uniform.
    \end{proof}

We will use the following definition of $\eta$-layered set sequences to characterize the sequence of information we revealed during our removal process.

\begin{definition}[$\eta$-layered set sequence]
    Let  $(A_1,\dots, A_r)$ be an ordered sequence of sets, and $\eta\geq 0$. We say that $(A_1,\dots, A_r)$ is \emph{$\eta$-layered} if for all $1\leq j\leq r-1$, we have
    \[\abs{A_j\setminus (A_{j+1}\cup\dots\cup A_{r})}\geq \eta.\]
\end{definition}

The following lemma formalizes the main property of $\eta$-layered set sequences we will use, which is that the parities of Bernoulli random variables indexed by elements in these sets are very close to being uniformly distributed.

\begin{lemma}\label{lemma:AlmostUniform}
    Let $p\in(0,1)$ and $X_1,\dots,X_t$ be independent Bernoulli random variables with probability $p$. Let $p^*=\min\{p,1-p\}\in(0,1/2]$.
    Let $(A_1,\dots, A_r)$ be an $\eta$-layered set sequence, with $A_j\subseteq [t]$ for every $j$. For every $1\leq j\leq r$, define the random variable
    \[Y_j = \sum_{i\in A_j}X_i  \pmod 2,\]
     i.e., the parity of $\sum_{i\in A_j} X_i$. Suppose $re^{-2\eta p^*}<1$. Then $(Y_1,\dots,Y_r)$ is $re^{-2\eta p^*}$-uniform.
\end{lemma}

\begin{proof}
We prove by induction on $r$. 
For $r=1$, since $|A_1|\ge \eta$, the statement follows from \cref{lem:parity-close-to-uniform}. Inductively, suppose $r\geq 2$ and the lemma is true for all $1\leq r' \leq r-1$. Let $(A_1,\dots, A_r)$ be an $\eta$-layered sequence with $A_j\subseteq [t]$ for every $j$. Consider any $s_1,\dots, s_r\in \{0,1\}$, and let $\cE$ denote the event that $Y_j = s_j$ for all $j\in \{2,\dots, r\}$. Since $(A_2,\dots, A_r)$ is $\eta$-layered, we know from the inductive hypothesis that
\[
\PP[\cE]\in \lbr{\lpr{1-(r-1)e^{-2 \eta p^*}}\frac{1}{2^{r-1}},\,  \lpr{1-(r-1)e^{-2\eta p^*}}^{-1}\frac{1}{2^{r-1}}}.
\]
Thus, with $T := A_2\cup \dots \cup A_r$, we have
    \begin{align*}
        \PP\lbr{Y_1 = s_1,\dots, Y_r = s_r}
       & = \PP\lbr{Y_1 = s_1,\, \cE}\\
        &= \PP\lbr{Y_1 = s_1 \mid \cE}\cdot \PP[\cE]\\
        &= \PP\lbr{\sum_{i\in A_1\setminus T}X_i = \sum_{i\in A_1\cap T}X_i + s_1 \pmod 2\,\middle\vert\,  \cE}\cdot \PP[\cE]\\
        &\ge \frac{1}{2}\lpr{1-e^{-2\eta p^*}} \lpr{1-(r-1)e^{-2\eta p^*}}\frac{1}{2^{r-1}}\\
        &\ge  \lpr{1-re^{-2\eta p^*}}\frac{1}{2^r}.
    \end{align*}
    The second to last inequality follows from \cref{lem:parity-close-to-uniform}, as  $\{X_i :i\in A_1\setminus T\}$ is independent from $\mathcal E$ and is a set of at least $\eta$  i.i.d. Bernoulli variables with probability $p$. By the same argument, we also know that  $\PP\lbr{Y_1 = s_1,\dots, Y_r = s_r}\leq \lpr{1-re^{-2\eta p^*}}^{-1}\frac{1}{2^r}$. Thus $(Y_1,\dots,Y_r)$ is $re^{-2\eta p^*}$-uniform.
\end{proof}

    When we apply \cref{lemma:AlmostUniform} in our problem, each of $X_1,\dots,X_t$ will be an indicator variable of whether the edge $uv$ is present in $G$ for a pair of vertices $u$ and $v$; the sets $A_1,\dots,A_r$ will be sets of vertex pairs in $\binom{V}{2}$.

Sometimes we would like to show a result like \cref{lemma:AlmostUniform} where the ordered set sequence $(A'_1,\dots,A'_r)$ is not $\eta$-layered at the beginning. The following observation shows that we may first apply an invertible linear transformation $\FF_2^r\to\FF_2^r$ to the sets by viewing each of them as a vector in $\FF_2^r$, and apply \cref{lemma:AlmostUniform} to the resulting collection of sets.

\begin{observation}\label{obs:LinearTrans}
    Suppose $X_1,\dots,X_t\in\{0,1\}$.
    Let $A_1,\dots,A_r\subseteq [t]$. For every $1\leq j\leq r$, let 
    \[Y_j = \sum_{i\in A_j}X_i  \pmod 2.\]
    Let $T:\FF_2^r\rightarrow\FF_2^r$ be an invertible linear transformation and let $A'_1,\dots,A'_r$ be the sets given by 
    \[(\mathbbm{1}_{A'_1}(i),\dots,\mathbbm{1}_{A'_r}(i))=T(\mathbbm{1}_{A_1}(i),\dots,\mathbbm{1}_{A_r}(i))\]
    for all $i\in [t]$. For every $1\leq j\leq r$, let
    \[Y'_j = \sum_{i\in A'_j}X_i  \pmod 2.\]
    Then for all $s_1,\dots,s_r\in\{0,1\}$, the event $(Y_1,\dots,Y_r)=(s_1,\dots,s_r)$ is the same as the event $(Y'_1,\dots,Y'_r)=T(s_1,\dots,s_r)$.
\end{observation}

Motivated by this observation, we have the following definition and corollary.
\begin{definition}[$\eta$-transformed set sequence]\label{def:linear-transform}
    We say that an ordered set sequence $(A'_1,\dots,A'_r)$ is $\eta$-transformed if there exists an $\eta$-layered set sequence $(A_1,\dots,A_r)$ and invertible linear transformation $T:\FF_2^r\to\FF_2^r$ such that 
    \[(\mathbbm{1}_{A'_1}(i),\dots,\mathbbm{1}_{A'_r}(i))=T(\mathbbm{1}_{A_1}(i),\dots,\mathbbm{1}_{A_r}(i))\]
    holds for all $i$.
\end{definition}
\begin{corollary}\label{cor:transformed}
    Under the same setting as in \cref{lemma:AlmostUniform}, if $(A_1,\dots,A_r)$ is $\eta$-transformed instead of $\eta$-layered, we can still conclude that $(Y_1,\dots,Y_r)$ is $re^{-2\eta p^*}$-uniform.
\end{corollary}

The following three lemmas will be used to prove a number of results in \cref{sec:uw-removal,sec:algorithm}. These results imply that, after revealing some partial information in the removal procedure, we still have enough randomness in the remaining graph. In particular, since we need to reveal the degree parities of certain vertices during the procedure, we will use \cref{lem:Y-epsUniform,lem:evensumepsUniform} to show that the resulting distribution of the graph is close to the original distribution. In the main proof, these two lemmas will be applied with the Bernoulli random variables $X_1,\dots,X_t$ being the indicator variables of vertex pairs in $\binom{V}{2}$ (or potential edges as defined later) and the parity vectors $\vec Y,\vec Z$ being the degree parities of vertices.

\begin{lemma}\label{lem:Y-epsUniform}
    Let $X_1,\dots,X_t$ be Bernoulli random variables (not necessarily independent or having the same law). Let $1\le k<t$ and $T:=\{k+1,\dots,t\}$.
    Suppose the random variables \linebreak$\vec{X} :=(X_1,\dots,X_{k})$ and $\vec X':=(X_{k+1},\dots,X_t)$ are independent.
    
    For subsets $A_1,\dots,A_r\subseteq [t]$, define the random variables $\vec Y=(Y_1,\dots,Y_r)\in\bin^r$ and \linebreak$\vec Z=(Z_1,\dots,Z_r)\in \bin^r$ by
    \[
    Y_j=\sum_{i\in A_j}X_i\pmod{2},\quad Z_j=\sum_{i\in A_j\cap T}X_i\pmod{2}\qquad \text{for $j=1, \dots, r$.}
    \]
    If $\vec{Z}$ is $\eps$-uniform, then we have the following:
    \begin{enumerate}
        \item $\vec{Y}$ is $\eps$-uniform.  
        \item If $\eps< 1/2$, then $\vec X$ is $2\veps$-affected by $\vec Y$.
    \end{enumerate}
\end{lemma}

\begin{proof} Since $\vec Z$ is a function of $\vec X'$, and $\vec X$ and $\vec X'$ are independent, we know that $\vec X$ and $\vec Z$ are independent. Fix $\vec x\in\{0,1\}^k$ and $\vec y\in\{0,1\}^r$. Observe that given $\vec X=\vec x$, the statement that $\vec Y=\vec y$ is equivalent to the statement that $\vec Z=\vec z$, where $\vec z=\vec z(\vec x,\vec y)\in\{0,1\}^r$ is the vector defined by 
\[
z_j:=y_j-\sum_{i\in A_j\cap [k]}x_i\pmod 2 \qquad \forall j\in[r].
\]
Thus for every $\vec y\in\{0,1\}^r$, we have
    \begin{align*}
        \PP[\vec{Y} = \vec{y}] &= \sum_{\vec{x} \in \bin^k}\PP[\vec{Y} = \vec{y} \mid \vec{X} = \vec{x}]\cdot \PP[\vec{X} = \vec{x}]=\sum_{\vec{x}\in \bin^k}\PP[\vec{Z} = \vec z \mid \vec{X} = \vec{x}]\cdot \PP[\vec{X} = \vec{x}]\\
        &= \sum_{\vec{x}\in \bin^k}\PP[\vec{Z} = \vec{z}]\PP[\vec{X} = \vec{x}]\ge (1-\eps)\frac{1}{2^r}\sum_{\vec{x} \in \bin^k}\PP[\vec{X} = \vec{x}] = (1-\eps)\frac{1}{2^r},
    \end{align*}
    where in the second to last step we used the assumption that $\vec{Z}$ is $\eps$-uniform. By the same argument, we have the upper bound  $\PP[\vec{Y} = \vec{y}]\leq (1-\eps)^{-1}\frac{1}{2^r}$ as well. Since this holds for all $\vec y\in\{0,1\}^r$, we get that  $Y_1,\dots, Y_r$ are $\eps$-uniform. This proves (1).

We now prove (2). For every $\vec x\in\bin^k$ and $\vec y\in\bin^r$, with $\vec z=\vec z(\vec x,\vec y)$ defined above, we have
    \begin{align*}
        \PP[\vec{X} = \vec{x} \mid \vec{Y} = \vec{y}] = \frac{\PP[\vec{X} = \vec{x}, \vec{Y}= \vec{y}]}{\PP[\vec{Y} = \vec{y}]} = \frac{\PP[\vec{X} = \vec{x}, \vec{Z} = \vec{z}]}{\PP[\vec{Y} = \vec{y}]}=\frac{\PP[\vec{X} = \vec{x}]\cdot \PP[ \vec{Z} = \vec{z}]}{\PP[\vec{Y} = \vec{y}]}.
    \end{align*}
    By (1), $\vec{Y}$ is $\eps$-uniform, so we can conclude that 
    \[\frac{\PP[\vec{X} = \vec{x}]\cdot \PP[ \vec{Z} = \vec{z}]}{\PP[\vec{Y} = \vec{y}]}\ge \frac{(1-\eps)2^{-r}}{(1-\eps)^{-1}2^{-r}}\PP[\vec{X} = \vec{x}]\ge (1-2\eps)\PP[\vec{X} = \vec{x}].\]
    The upper bound follows symmetrically.
\end{proof}

We also need the following variant of \cref{lem:Y-epsUniform} for fix-parity distributions.

\begin{lemma}\label{lem:evensumepsUniform}
    Let $X_1,\dots,X_t$ be Bernoulli random variables (not necessarily independent or having the same law). 
    Let $1\le k<t$ and $T:=\{k+1,\dots,t\}$.
    Suppose the random variables \linebreak$\vec{X} :=(X_1,\dots,X_{k})$ and $\vec X':=(X_{k+1},\dots,X_t)$ are independent.
    
    For subsets $A_1,\dots,A_r\subseteq [t]$, define the random variables $\vec Y=(Y_1,\dots,Y_r)\in\bin^r$ and \linebreak$\vec Z=(Z_1,\dots,Z_r)\in \bin^r$ by
    \[
    Y_j=\sum_{i\in A_j}X_i\pmod{2},\quad Z_j=\sum_{i\in A_j\cap T}X_j\pmod{2}\qquad \text{for $j=1, \dots, r$.}
    \]
    Finally, let $W=\sum_{j=1}^r Y_j\pmod{2}$.
    Suppose $\vec{Z}$ is fix-parity $\eps$-uniform, with $\veps< 1/2$ and parity $s=\sum_{j=1}^rZ_j\pmod{2}$. Then $(\vec{X}\mid \vec{Y})$ is $2\veps$-close to $(\vec{X}\mid W)$.
\end{lemma}
\begin{proof}
    Note that 
    \[W\equiv\sum_{j=1}^r\sum_{i\in A_j\setminus T}X_i+\sum_{j=1}^rZ_j\equiv\sum_{j=1}^r\sum_{i\in A_j\setminus T}X_i+s\pmod{2}.\]
    Note that $s$ is deterministic, so $W$ only depends on $\vec{X}$. Thus the concatenated vector $(\vec{X}, W)$ is independent from $\vec{X}'$. 
    
    For each $w\in\{0,1\}$, we apply \cref{lem:Y-epsUniform} with $(Y_1,\dots,Y_{r-1})$ in place of $\vec{Y}$, $(\vec{X}\mid W=w)$ in place of $\vec{X}$, and $\vec X'$ as it is. We already showed that $(\vec{X}\mid W=w)$ and $\vec X'$ are independent. Furthermore, since $\vec Z=(Z_1,\dots,Z_r)$ is fix-parity $\eps$-uniform, it follows that $(Z_1,\dots,Z_{r-1})$ is $\eps$-uniform.
    Thus, we may conclude from \cref{lem:Y-epsUniform} that $(\vec{X}\mid W=w)$ is $2\veps$-affected by $(Y_1,\dots,Y_{r-1})$. This shows that $((\vec X\mid W=w)\mid Y_1,\dots,Y_{r-1})$ is $2\veps$-close to $(\vec{X}\mid W=w)$. From the definition of $\varepsilon$-closeness, this implies that $(\vec X\mid W, Y_1,\dots,Y_{r-1})$ is $2\veps$-close to $(\vec{X}\mid W)$. Since $\vec Y$ and $(W,Y_1,\dots,Y_{r-1})$ uniquely determine each other, we get that $(\vec{X}\mid \vec Y)$ is $2\veps$-close to $(\vec{X}\mid W)$.
\end{proof}

For convenience, we define the following notion for describing sets with almost equal sizes.
\begin{definition}[Balanced sets]
Let $A_1,\dots,A_r$ be finite sets. We say that $A_1,\dots,A_r$ are balanced if $\abs{A_i},\abs{A_j}$ differ by at most $1$ for all $i,j\in [r]$.
\end{definition}

Finally, we have the following lemma that shows the degree parities of vertices in a $p$-random bipartite graph are close to being fix-parity uniform.

\begin{lemma}\label{lem:bipartite}
    Fix $p\in (0,1)$ and set $p^*=\min(p,1-p)$. There exists $\eta_0=\eta_0(p)$ such that the following holds. For any sufficiently large integer $\eta>\eta_0$, let $G=(V,E)$ be a random bipartite graph with bipartition $V = A\sqcup B$ such that $\eta\leq |A|, |B|\leq e^{\eta(p^*)^2/10}$, and every edge between $A, B$ is included independently with probability $p$. For $v\in V$ and $C\subseteq V$, let $\parity(v,C)\in\{0,1\}$ denote the parity of the number of edges of $G$ between $v$ and $C$. Then $((\parity(a, B))_{a\in A}, (\parity(b,A))_{b\in B}) \in\{0,1\}^{|A|+|B|}$ is fix-parity $e^{-\eta(p^*)^2/30}$-uniform.
\end{lemma}

To prove \cref{lem:bipartite}, first note that $\sum_{a\in A}\parity(a,B)+\sum_{b\in B}\parity(b,A)$ is always even, so it suffices to check that with $c:=(p^*)^2/30$, we have
\[\PP\left[((\parity(a, B))_{a\in A}, (\parity(b,A))_{b\in B}) =s\right]\in\left[(1-e^{-c\eta})\cdot \frac{1}{2^{\abs{A}+\abs{B}-1}},(1-e^{-c\eta})^{-1}\cdot \frac{1}{2^{\abs{A}+\abs{B}-1}}\right]\]
for all $s = ((s_a)_{a\in A}, (s_b)_{b\in B})\in\{0,1\}^{|A|+|B|}$ with $\sum_{a\in A}s_a+\sum_{b\in B}s_b=0\pmod{2}$.

We first prove the case where the sizes of $A,B$ are unbalanced.

\begin{lemma}\label{lem:bip_unbalanced}
     Assume the setup of \cref{lem:bipartite} except instead of $\eta\leq |A|, |B|\leq e^{\eta(p^*)^2/10}$, we suppose that  $$5|B|/p^*\leq \abs{A}\leq e^{\eta(p^*)^2/10} \quad \text{and }\quad \abs{B}\geq \eta p^*/10.$$ Then $((\parity(a, B))_{a\in A}, (\parity(b,A))_{b\in B}) $ is fix-parity $e^{-\eta(p^*)^2/20}$-uniform.
\end{lemma}
\begin{proof}
    Suppose $\eta$ is sufficiently large so that $m:=|B|\geq \eta p^*/10\geq  2$. In the following, we will also assume $\eta$ is large enough so that some inequalities hold. Observe that, since the sets $a\times B$ for all $a\in A$ are pairwise disjoint, by \cref{lemma:AlmostUniform}, $(\parity(a, B))_{a\in A}$ is $|A|e^{-2m p^*}$-uniform. Let $\cP_A := (\parity(a, B))_ {a\in A}$ and $\cP_B := (\parity(b, A))_ {b\in B}$.
    
    Fix any ordering $b_1,\dots, b_{m}$ of the vertices in $B$.
    For convenience, let $P_i :=  \parity(b_i, A)$ for every $i\in[m]$ and let $N_S(v)$ denote the set of neighbors of $v$ in $S$. 
    Consider a balanced partition $A = A_0\sqcup A_1$. For every $ i\in [m-1]$, let $\cE_i$ denote the event that $N_{A_{j}}(b_{i}) \ne N_{A_{j}}(b_{i+1})$ where $j=i\pmod{2}$; let $\overline{\cE}_i$ denote its complement event $N_{A_{j}}(b_{i}) = N_{A_{j}}(b_{i+1})$ (that is, $b_i$ and $b_{i+1}$ have identical neighborhoods in $A_j$). We first prove that it is unlikely for any $\overline{\cE}_i$ to occur.

    \begin{claim}\label{claim:bipartite-pr-oneequal}
   For every $s_A\in\{0,1\}^{|A|}$, we have
        \[\PP[\exists i\in [m-1] \text{ such that } \overline{\cE}_i \mid \cP_A=s_A] \le 2m e^{-\abs{A} p^*/3}.\]
    \end{claim}
    
\begin{proof*}
We first show that for every $i\in [m-1]$, whether $\cE_i$ is true is $2|A|e^{-2(m-2) p^*}$-affected by $\cP_A$.
    Again, since the sets $a\times (B\setminus \{b_{i}, b_{i+1}\})$ for all $a\in A$ are pairwise disjoint, we know from \cref{lemma:AlmostUniform} that $(\parity(a, B\setminus \{b_{i}, b_{i+1}\})_{a\in A}$ is $|A|e^{-2(m-2) p^*}$-uniform.
    Thus we can apply \cref{lem:Y-epsUniform} where
    \begin{enumerate}
        \item $\vec X$ is the subgraph of $G$ on $A\times \{b_{i},b_{i+1}\}$,
        \item $\vec X'$ is the subgraph of $G$ on $A\times (B\setminus\{b_{i},b_{i+1}\})$,
        \item the sets $A_1,\dots, A_r$ are $a\times B$ for all $a\in A$, such that $\vec Y=(\parity(a, B))_{a\in A}=\mathcal P_A$ and $\vec Z=(\parity(a, B\setminus \{b_{i}, b_{i+1}\}))_{a\in A}$.
    \end{enumerate}
    By \cref{lem:Y-epsUniform}, the subgraph of $G$ on $A\times \{b_{i},b_{i+1}\}$ is $2|A|e^{-2(m-2) p^*}$-affected by $\mathcal P_A$. Since whether $\cE_i$ holds is completely determined by this subgraph, it is $2|A|e^{-2(m-2) p^*}$-affected by $\mathcal P_A$ as well.
Thus for all $i\in [m-1]$ and $j=i\pmod{2}$, we have
    \begin{align*}
        \PP[\overline{\cE}_i \mid \cP_A=s_A] &\le \lpr{1-2|A|e^{-2(m-2) p^*}}^{-1}\PP[\overline{\cE}_i].
    \end{align*}
    Since
    \begin{align*}
        \PP[\overline{\cE}_i]=\PP[N_{A_j}(b_i)=N_{A_j}(b_{i+1})]&\leq \max_{0\le x\le |A_{j}|}p^x(1-p)^{|A_{j}| - x}=(1-p^*)^{|A_j|}\leq  e^{-\abs{A} p^*/3},
    \end{align*}
   we get that
   \begin{align*}
       \PP[\overline{\cE}_i \mid \cP_A=s_A] &\le \lpr{1-2|A|e^{-2(m-2) p^*}}^{-1}e^{-\abs{A} p^*/3}\leq 2e^{-\abs{A} p^*/3}.
   \end{align*}
   Here we used the fact that $\abs{A}\leq  e^{\eta(p^*)^2/10}\leq e^{2(m-2)p^*}/4$, as $\eta$ is large enough. A union bound over all $i\in [m-1]$ finishes the proof.
\end{proof*}

    Next, we will show that, conditioning on $\cP_A$ and $\cE_1,\dots,\cE_{m-1}$, the parities $(P_i)_{i\in|B|}$ are fix-parity uniform.
    \begin{claim}\label{claim:bipartite-pair}
        For any $s_A:=(s_a)_{a\in A}\in\{0,1\}^{|A|}$ and  $s_1,\dots, s_m\in \{0,1\}$ satisfying $\sum_{i=1}^ms_i=\sum_{a\in A}s_a$, we have
        \[\PP[P_i = s_i \,\, \forall i\in [m] \mid \cE_1,\dots, \cE_{m-1}, \cP_A=s_A] = \frac{1}{2^{m-1}}.\]
    \end{claim}

    \begin{proof*}
    It suffices to show that for each $i\in [m-1]$, there is a bijection $f_i$ between bipartite graphs on $A\sqcup B$ satisfying $\cE_1,\dots,\cE_{m-1}$ such that the following three conditions hold:
        \begin{enumerate}
            \item $f_i$ preserves the total number of edges;
            \item $f_i$ preserves $\cP_A=(\parity(a,B))_{a\in A}$ and $(\parity(b,A))_{b\in B\setminus\{ b_i,b_{i+1}\}}$;
            \item $f_i$ flips both $\parity(b_i,A)$ and $\parity(b_{i+1},A)$.
        \end{enumerate}

        We define the bijection $f_i$ as follows. Let $j=i\pmod{2}$ and fix an ordering $A_{j}=\{a_1,\dots,a_{\abs{A_j}}\}$. Assume $H$ is a bipartite graph on $A\sqcup B$ satisfying $\cE_1,\dots,\cE_{m-1}$. 
        Since $H$ satisfies $\cE_i$, we can pick a vertex $a_k$ such that exactly one of $b_ia_k$ and $b_{i+1}a_k$ is an edge in $H$. We pick $k_H$ to be the smallest such $k$.
        Define $f_i(H)$ to be the graph obtained from $H$ by switching $b_ia_{k_H}$ and $b_{i+1}a_{k_H}$. It is easy to verify that $f_i$ satisfies the above three conditions. In particular, $f_i$ is a bijection since $f_i(f_i(H))=H$ for all $H$, which follows from the fact that $k_H=k_{f_i(H)}$.
    \end{proof*}
    
    Now we proceed to proving \cref{lem:bip_unbalanced}. For any $s=((s_a)_{a\in A}, (s_b)_{b\in B})\in \{0,1\}^{|A|+|B|}$ with $\sum_{a\in A}s_a+\sum_{b\in B}s_b\equiv 0\pmod{2}$, letting $s_A=(s_a)_{a\in A}$ and $s_B=(s_b)_{b\in B}$, we have
    \begin{align*}
            \PP[\cP_A=s_A, \,\cP_B=s_B ]=\PP[\cP_B=s_B  \mid \cP_A=s_A]\cdot\PP[\cP_A=s_A].
        \end{align*}
        Since $\cP_A$ is $\abs{A}e^{-2m p^*}$-uniform, we have
        \[\PP[\cP_A=s_A]\in \lbr{(1-\abs{A}e^{-2m p^*})\frac{1}{2^{\abs{A}}},(1-\abs{A}e^{-2m p^*})^{-1}\frac{1}{2^{\abs{A}}}}.\]
        Also by \cref{claim:bipartite-pr-oneequal} and \cref{claim:bipartite-pair}, the probability $\PP[\cP_B=s_B \mid \cP_A=s_A]$ is at least 
        \begin{align*}
            &\PP[\cP_B=s_B, \cE_1,\dots,\cE_{m-1} \mid\cP_A=s_A]\\
            &= \PP[\cP_B=s_B\mid  \cE_1,\dots,\cE_{m-1}, \cP_A=s_A]\cdot\PP[ \cE_1,\dots,\cE_{m-1} \mid\cP_A=s_A]\\
            &=\frac{1}{2^{m-1}}\PP[ \cE_1,\dots,\cE_{m-1} \mid\cP_A=s_A]
            \geq (1-2me^{-\abs{A} p^*/3})\frac{1}{2^{m-1}}
        \end{align*}
        and is at most 
        \begin{align*}
            &\PP[\cP_B=s_B, \cE_1,\dots,\cE_{m-1} \mid\cP_A=s_A]+\PP[\exists i\in [m-1] \text{ such that } \overline{\cE}_i \mid \cP_A=s_A]\\
            &= \frac{1}{2^{m-1}}\PP[ \cE_1,\dots,\cE_{m-1} \mid\cP_A=s_A]+(1-\PP[ \cE_1,\dots,\cE_{m-1} \mid\cP_A=s_A])\\
            &\leq (1-2me^{-\abs{A} p^*/3})\frac{1}{2^{m-1}}+ 2me^{-\abs{A} p^*/3}.
        \end{align*}
        Since $m=\abs{B}\leq \abs{A}p^*/5$, we know that when $m$ is large enough, we have
        \[(1-2me^{-\abs{A} p^*/3})\frac{1}{2^{m-1}}\geq (1-e^{-\abs{A} p^*/6})\frac{1}{2^{m-1}}\]
        and
        \begin{align*}
            (1-2me^{-\abs{A} p^*/3})\frac{1}{2^{m-1}}+2me^{-\abs{A} p^*/3}
            &=(1-2me^{-\abs{A} p^*/3}+2me^{-\abs{A} p^*/3+(m-1)\ln 2})\frac{1}{2^{m-1}}\\
            &\leq (1-e^{-\abs{A} p^*/6})^{-1}\frac{1}{2^{m-1}}.
        \end{align*}
        Thus, we showed that $\PP[\cP_A=s_A, \,\cP_B=s_B ] $ is in 
        \[\lbr{(1-e^{-\abs{A} p^*/6})(1-\abs{A}e^{-2mp^*})\frac{1}{2^{\abs{A}+m-1}}, (1-e^{-\abs{A} p^*/6})^{-1}(1-\abs{A}e^{-2mp^*})^{-1}\frac{1}{2^{\abs{A}+m-1}}}.\] 
        Since $(1-e^{-\abs{A} p^*/6})(1-\abs{A}e^{-2mp^*})\geq 1-e^{-\eta(p^*)^2/20}$ when $\eta$ is large enough, we conclude that $(\parity(a, B), \parity(b,A))_{a\in A, b\in B}$ is fix-parity $e^{-\eta(p^*)^2/20}$-uniform.
\end{proof}

\begin{proof}[Proof of \cref{lem:bipartite}]
Now, we may prove the general case for \cref{lem:bipartite}. We may assume without loss of generality that $\eta\leq\abs{B}\leq\abs{A}\leq e^{\eta(p^*)^2/10}$. Partition $B$ into $k=\lfloor6/p^*\rfloor$ balanced sets $B_1,\dots,B_k$. Then we may apply \cref{lem:bip_unbalanced} to the bipartite graph between $A$ and $B_i$ for each $i\in [k]$ and conclude that $((\parity(a, B_i)_{a\in A}, (\parity(b,A))_{b\in B_i})$ is fix-parity $e^{-\eta(p^*)^2/20}$-uniform. Since the subgraphs of $G$ restricted on $A\times B_1,\dots,A\times B_k$ are independent, we know that for all $s^{(1)}_A,\dots,s^{(k)}_{A}\in\{0,1\}^{|A|}$ and $s_{B_1}\in\{0,1\}^{|B_1|} ,\dots,s_{B_k}\in\{0,1\}^{|B_k|}$ such that every $s^{(i)}_A$ and $s_{B_i}$ have the same parity sum, we have 
\begin{align*}
    &\PP[(\parity(a, B_i)_{a\in A}=s^{(i)}_A,\, (\parity(b,A))_{b\in B_i}=s_{B_i} \text{ for all }i\in[k]]\\
    &=\prod_{i=1}^k \PP[(\parity(a, B_i)_{a\in A}=s^{(i)}_A,\, (\parity(b,A))_{b\in B_i}=s_{B_i}]\\
    &\in \left[(1-e^{-\eta(p^*)^2/20})^k\frac{1}{2^{k|A|+\sum_{i=1}^k |B_i|-k}}, (1-e^{-\eta(p^*)^2/20})^{-k}\frac{1}{2^{k|A|+\sum_{i=1}^k |B_i|-k}}\right].
\end{align*}
Observe that for all $s_A\in\{0,1\}^{|A|}$ and $s_B=(s_{B_i})_{i\in[k]}\in\{0,1\}^{|B|}$ with the same parity sum, there are $2^{(|A|-1)(k-1)}$ ways to choose $s_A^{(1)},\dots,s_A^{(k)}\in\{0,1\}^{|A|}$ such that $s_A=\sum_{i=1}^k s_A^{(i)}$, and every $s_A^{(i)}$ and $s_{B_i}$ have the same parity sum. Indeed, one can pick $s_A^{(1)},\dots,s_A^{(k-1)}$ arbitrarily as long as the parity condition holds, and these uniquely fix $s_A^{(k)}=s_A-\sum_{i=1}^{k-1} s_A^{(i)}$.

Thus, we get that for every $s_A$ and $s_B$, we have
\begin{align*}
    &\PP[(\parity(a, B)_{a\in A}=s_A,\, (\parity(b,A))_{b\in B}=s_{B}]\\
    &=\sum_{(s_A^{(i)})_{i\in[k]}}\PP[(\parity(a, B_i)_{a\in A}=s^{(i)}_A,\, (\parity(b,A))_{b\in B_i}=s_{B_i} \text{ for all }i\in[k]]\\
    &\in
    \left[2^{(|A|-1)(k-1)}(1-e^{-\eta(p^*)^2/20})^k\frac{1}{2^{k|A|+\sum_{i=1}^k |B_i|-k}}, 2^{(|A|-1)(k-1)}(1-e^{-\eta(p^*)^2/20})^{-k}\frac{1}{2^{k|A|+\sum_{i=1}^k |B_i|-k}}\right]\\
    &=\left[(1-e^{-\eta(p^*)^2/20})^k\frac{1}{2^{|A|+|B|-1}}, (1-e^{-\eta(p^*)^2/20})^{-k}\frac{1}{2^{|A|+|B|-1}}\right].
\end{align*}
Thus, we get that $((\parity(a, B))_{a\in A}, (\parity(b,A))_{b\in B})$ is fix-parity $e^{-\eta(p^*)^2/30}$-uniform, as
\[(1-e^{-\eta(p^*)^2/20})^k\geq 1-ke^{-\eta(p^*)^2/20}\geq 1-e^{-\eta(p^*)^2/30},\] 
when $\eta$ is large enough.
\end{proof}

\section{$(U,W)$-removal}\label{sec:uw-removal}

The following \cref{sec:uw-removal} and \cref{sec:algorithm} will be devoted to proving \cref{thm:recurrence}. To prove \cref{thm:recurrence}, we show that given any $\alpha$-partially revealed $p$-graph $G$, with high probability we can sequentially remove $m:=3n/4\pm o(n)$ vertices $v_1,\dots,v_{m}$ from the graph $G$ so that an even number of edges is removed every time a vertex is removed.
Furthermore, the procedure used to choose these vertices $v_1,\dots,v_{m}$ does not reveal too much information about $G$. This means that the remaining graph after the removal is close to an $\alpha$-partially revealed $p$-graph on which we can apply induction.

As mentioned in the \cref{sec:proof-strategy}, the above process is what we will call the $(U,W)$-removal. We remark that to get the $f_\alpha(n)\leq e^{-\Theta(n^{1/2-\alpha})}$ upper bound, we do need a quadratic term $\max_{n'\in [n/4-o(n),n/4+o(n)]}f_\alpha(n')^2$ on the right hand side of \cref{thm:recurrence}. This in practice means that during the induction, we need to perform $(U,W)$-removal on $G$ twice so that the probability that neither of them succeeds is very low.

\subsection{Notations for random graphs}
We first introduce the notations we will use in \cref{sec:algorithm,sec:uw-removal}.
From now on, we fix $0 < \alpha < 1/2$ and $p\in (0,1)$, so for convenience we drop $\alpha$ from our notations (e.g., we say $G$ is a partially revealed $p$-graph to mean that $G$ is an $\alpha$-partially revealed $p$-graph). We also use $f(n)$ to denote $f_\alpha(n)$ as the context is clear from now on.

Let $G = (V,E)$ be a random graph on $n$ vertices. We will use the following notations. We call the 2-sets in $\binom{V}{2}$ \emph{potential edges} to indicate that $\{u,v\}$ may or may not be an edge in $G$. Therefore, $\binom{A}{2}$ is the collection of all potential edges within $A\subseteq V$. Also, we use $S(A,B)=\{\{a,b\}\mid a\in A,b\in B,a\neq b\}$ to denote all the potential edges between $A,B\subseteq V$. If $A=\{v\}$, we write $S(v,B)$ instead of $S(\{v\},B)$ to denote the star of potential edges centered at $v$. If $B=V(H)$ for a graph $H$, we write $S(A,H)$ instead of $S(A,V(H))$. 

For a set of potential edges $S\subseteq {\binom{V}{2}}$, we use $\parity_G(S)\in\{0,1\}$ to denote the parity of the number of edges in $S$, i.e., the parity of $\abs{E(G)\cap S}$. When $S=S(A,B)$, we write $\parity_G(A,B)$ instead of $\parity_G(S(A,B))$. Similarly, when $S=\binom{A}{2}$, we write $\parity_G(A)$ instead of $\parity_G\lpr{\binom{A}{2}}$. For any $S\in \binom{V}{2}$, when we say that we reveal the information of $G[S]$, it means that we condition on any feasible assignment of edges and non-edges to all the potential edges of $S$. Similar as before, when $S=\binom{A}{2}$, we write $G[A]$ instead of $G[S]$; when  $S=S(A,B)$, we write $G[A,B]$  instead of $G[S]$. Note that in this definition, the notation $G[A]$ agrees with the usual definition where $G[A]$ is the induced subgraph on $A$. Also, when we say that we reveal the information of $\parity_G(S)$, it means that we condition on any feasible assignment of $\parity_G(S)$.

Let $\mathcal{I}\subseteq 2^{\binom{V}{2}}$ be a family of sets of potential edges, and $\Sigma\subseteq\binom{V}{2}$. We let $G[\mathcal{I}]=\{G[S]\mid S\in\mathcal{I}\}$, $\parity_G(\mathcal{I})=\{\parity_G(S)\mid S\in \mathcal{I}\}$, and $\mathcal{I}|_\Sigma=\{S\cap\Sigma\mid S\in\mathcal{I}\}$. When we say that we reveal $G[\mathcal{I}]$, it means that we reveal $G[S]$ for all $S\in\mathcal{I}$. Similarly, when we say that we reveal $\parity_G(\mathcal{I})$, it means that we reveal $\parity_G(S)$ for all $S\in\mathcal{I}$. When $G$ is clear from  context, we will write $\parity$ instead of $\parity_G$.

\subsection{$(U,W)$-removal}\label{subsec:lr-removal}

We define the process of $(U,W)$-removal as follows.

\begin{definition}[$(U,W)$-removal]\label{def:lrremoval}\quad

    \textbf{Input}: 
     a partially revealed graph $G=(V,E)$ with revealed part $A=\{a_1.\dots,a_{\abs{A}}\}$, and a partition $V\setminus A = U\sqcup W$, where $U$ is equipped with an ordering $U=\{u_1,\dots,u_{|U|}\}$, and $W$ is equipped with a partition $W=W_{\#}\sqcup W_1\sqcup \dots \sqcup W_s$

    \textbf{Output}: If the process succeeds, then it outputs a subset $V_W\subseteq W$ and an ordered sequence $R=(v_1,\dots, v_{|V\setminus V_W|})$ of $V\setminus V_W$ such that for every $1\leq i\leq |V\setminus V_W|$, $\parity(v_i, V\setminus \{v_1,\dots, v_{i-1}\}) = 0$. Otherwise the process fails and outputs nothing.

    \textbf{Maintained sets/pointers:}
        We use pointers $\ptrU,\ptrW$ to keep track of the number of vertices removed from $A\cup U$ and $W$ respectively. We let $R^{(\ptrU)}$ denote the set of vertices removed in the first $\ptrU$ rounds (i.e., after $\ptrU$ vertices are removed from $A\cup U$). We use $\mathcal I_A^{(\ptrU)}$, $\mathcal I_U^{(\ptrU)}$, $\mathcal I_W^{(\ptrU)}$ to denote the collections of potential edge sets whose parities are revealed and $\mathcal I_e^{(\ptrU)}$ to denote the potential edge sets revealed in the first $\ptrU$ rounds.
        
    \textbf{Procedure}: 
    
    Concatenate $A$ and $U$ to form a sequence $(\overline u_1,\dots,\overline u_{\abs{A}+\abs{U}}):=(a_1,\dots,a_{\abs{A}},u_1,\dots,u_{\abs{U}})$.

    Initialize $\ptrU=1$, $\ptrW= 1$, $R^{(0)}=\varnothing$, and $\mathcal I_A^{(0)},\mathcal I_U^{(0)},\mathcal I_W^{(0)},\mathcal I_e^{(0)}=\varnothing$.

    For $\ptrU=1,\dots,\abs{A}+\abs{U}$:
    \begin{enumerate}[1.]
        \item Reveal $\parity (\overline u_\ptrU,V\setminus R^{(\ptrU-1)})$. If $\overline u_\ptrU\in A$, set $\mathcal I_A^{(\ptrU)}:=\mathcal I_A^{(\ptrU-1)}\cup\{S(\overline u_\ptrU, V\setminus R^{(\ptrU-1)})\}$ and $\mathcal I_U^{(\ptrU)}:=\mathcal I_U^{(\ptrU-1)}$. If $\overline u_\ptrU\in U$, set $\mathcal I_U^{(\ptrU)}:=\mathcal I_U^{(\ptrU-1)}\cup\{S(\overline u_\ptrU, V\setminus R^{(\ptrU-1)})\}$ and $\mathcal I_A^{(\ptrU)}:=\mathcal I_A^{(\ptrU-1)}$.
        \item Given the value of $\parity (\overline u_\ptrU,V\setminus R^{(\ptrU-1)})$ revealed in step (1):
        \begin{itemize}
            \item If $\parity (\overline u_\ptrU,V\setminus R^{(\ptrU-1)})=0$:  set $R^{(\ptrU)}:=R^{(\ptrU-1)}\cup\{\overline u_\ptrU\}$ and set $v_{|R^{(\ptrU-1)}|+1}:=\overline u_\ptrU$. Set $\cI^{(\ptrU)}_e:=\cI^{(\ptrU-1)}_e$ and $\cI^{(\ptrU)}_W:=\cI^{(\ptrU-1)}_W$.
            \item If $\parity (\overline u_\ptrU,V\setminus R^{(\ptrU-1)})=1$: let $j=\ptrW\pmod s$ (recall that we use the convention \linebreak$s\pmod s=s$), reveal $G[\overline u_\ptrU, W_j\setminus R^{(\ptrU-1)}]$ and set $\mathcal I^{(\ptrU)}_e:=\mathcal I^{(\ptrU-1)}_e\cup\{S(\overline u_\ptrU, W_j\setminus R^{(\ptrU-1)})\}$. 
            Reveal $\parity(w,V\setminus R^{(\ptrU-1)})$ for all $w\in W_j\setminus R^{(\ptrU-1)}$ and set $\mathcal I^{(\ptrU)}_W:=\mathcal I^{(\ptrU-1)}_W\cup\{S(w,V\setminus R^{(\ptrU-1)})\mid w\in W_j\setminus R^{(\ptrU-1)}\}$.
            \begin{itemize}
                \item If there exists at least one vertex $w\in W_j\setminus R^{(\ptrU-1)}$ such that $\overline u_\ptrU w\in E$ and $\parity(w,V\setminus R^{(\ptrU-1)})=0$, then we pick one such $w$ and set $R^{(\ptrU)}:=R^{(\ptrU-1)}\cup\{w,\overline u_\ptrU\}$. More specifically, we may pick $w$ deterministically by prescribing an ordering of the $V$ and picking $w$ to be the smallest one in the ordering satisfying the conditions. Furthermore, we label $v_{|R^{(\ptrU-1)}|+1}:=w$, $v_{|R^{(\ptrU-1)}|+2}:=\overline u_\ptrU$, and increment $\ptrW$ by $1$.
                \item  If there is no such $w$, we terminate the process and say the process \emph{fails}.
            \end{itemize}
        \end{itemize}
    \end{enumerate}
    
    If the process does not fail, then we say the process \emph{succeeds} and output  $R:=R^{(\abs{A}+\abs{U})}$ and $V_W:=V\setminus R$. 
\end{definition}

    \begin{remark}
        In the above process,  we use the pointers $\ptrU, \ptrW$ to keep track of the number of vertices removed from $A\cup U$ and $W$ respectively. For every $\ptrU=1,\dots,\abs{A}+\abs{U}$, we aim to remove the vertex $\overline u_\ptrU$ at round $\ptrU$, and we use $R^{(\ptrU-1)}$ to denote the set of already removed vertices after round $\ptrU-1$. If $\parity (\overline u_\ptrU,V\setminus R^{(\ptrU-1)})=0$, then we can directly remove $\overline u_\ptrU$ and be done. If $\parity (\overline u_\ptrU,V\setminus R^{(\ptrU-1)})=1$, then we try to find some $w\in W_j$ where $j = \ptrW \pmod{s}$ so that we can first remove $w$ and then remove $\overline u_\ptrU$. Note that we require ourselves to loop through the blocks $W_1,\dots,W_s$ when we remove vertices from $W$. In other words, we want the $\ptrW$-th vertex removed from $W$ to lie in the block $W_{\ptrW\pmod s}$.
    \end{remark}

\subsection{The success probability of $(U,W)$-removal}\label{subsec:LRsuccess}

We will use the remainder of this section to prove the following proposition, which says that if the inputs are chosen appropriately, then with high probability the $(U,W)$-removal succeeds with the number of remaining vertices after the removal is concentrated around its expectation.

\begin{proposition}\label{prop:process-success-prob}
    Let $G=(V,E)$ be a partially revealed $p$-graph on $n$ vertices with $A\subseteq V$ being the revealed part.
    Let $U\sqcup W$ be a partition of $V\setminus A$ with $|U|, |W|=\frac{n}{2}\pm o(n)$. Suppose $U=\{u_1,\dots,u_m\}$ is an ordering of $U$, and $W=W_{\#}\sqcup W_1\sqcup \dots \sqcup W_s$ is a partition of $W$ such that $W_{\#},W_1,\dots,W_s$ are balanced with $s=\Theta(n^{1/2+\alpha})$. Then with probability $1-e^{-\Theta(n^{1/2-\alpha})}$, the $(U,W)$-removal succeeds with the output $V_W$ satisfying $|V_W| \in \lbr{\frac{n}{4}-o(n), \frac{n}{4}+o(n)}$.
\end{proposition} 

Consider any round $\ptrU$. Suppose we have successfully run $\ptrU-1$ rounds of the $(U,W)$-removal. 
    For convenience, we define $\mathcal{I}_A=\mathcal{I}^{(\ptrU)}_A$, $\mathcal{I}_U=\mathcal{I}^{(\ptrU)}_U$, $\mathcal{I}_W=\mathcal{I}^{(\ptrU)}_W$, $\mathcal{I}_e=\mathcal{I}^{(\ptrU)}_e$. Per our algorithm, these sets are well-defined no matter whether the $\ptrU$-th round succeeds or fails. 
    
    We emphasize that in the following analysis, all the hidden constants do not depend on $\ptrU$.

We know from \cref{def:lrremoval} that after the $\ptrU$-th round, we revealed $\parity(\mathcal{I}_A)$, $\parity(\mathcal{I}_U)$, $\parity(\mathcal{I}_W)$, $G[\mathcal{I}_e]$. Furthermore, letting $\mathcal{J}_A:=\{S(a,G)\mid a\in A\}$, we have previously revealed $G[A]$, $\parity(\mathcal{J}_A)$, $\parity(V)$ due to $G$ being a partially revealed graph.

Let $w_i$ denote the $i$-th vertex removed from $W$.
To compute the success probability of the $(U,W)$-removal, we first show the following lemma, which will imply that the parities we queried in the process are close to being uniformly distributed.
\begin{lemma}\label{lem:UWlayered}
    Let $\Sigma=\binom{W}{2}\cup S(A\cup U,W_{\#})$ and let $\mathcal{I}=\mathcal{J}_A\cup\mathcal{I}_U\cup\mathcal{I}_W\cup\{\binom{V}{2}\}$. The collection of sets of potential edges $\mathcal{I}|_{\Sigma}$ is $\Theta(n^{1/2-\alpha})$-transformed.
\end{lemma}
\begin{proof}
    
    We begin by examining what sets in $\mathcal I|_{\Sigma}$ 
    look like. We start with $\mathcal I_W|_{\Sigma}$.

    Consider any vertex $w\in W_j$. By definition of the $(U,W)$-removal, we know that every time a star centered at $w$ gets added to $\mathcal I_W$, we are at some round $\ptrU'$ with $\ptrW'=j\pmod s$, where $\ptrW'-1$ is the number of vertices already removed from $W$ before round $\ptrU'$. In this round, we added the star $S(w, V\setminus R^{(\ptrU'-1)})$ to $\mathcal I_W$. Clearly $S(w, V\setminus R^{(\ptrU'-1)})$ is disjoint from $S(A\cup U,W_{\#})$. Furthermore, since $R^{(\ptrU'-1)} \cap W=\{w_1,\dots,w_{\ptrW'-1}\}$, we get that
$$S(w, V\setminus R^{(\ptrU'-1)})\cap \binom{W}{2}= S(w, W\setminus\{w_1,\dots,w_{\ptrW'-1}\}).$$
Thus in this round, the star $S(w, W\setminus\{w_1,\dots,w_{\ptrW'-1}\})$ is added to $ \mathcal I_W|_{\Sigma}$.

For $w\in \bigcup_{j=1}^s W_j$, let $m(w)$ be the maximum index $\ptrW'$ such that $S(w, W\setminus\{w_1,\dots,w_{\ptrW'-1}\})\in \mathcal I_W|_{\Sigma}$. In other words, we have
\begin{align*}
    m(w)=\begin{cases}
        \ptrW' & \text{if } w=w_{\ptrW'}\in R,\\
        \max\{\ptrW'\mid \ptrW'\equiv j\pmod s,\, 1\leq \ptrW'\leq |W\cap R|\} & \text{if $w\in W_j\setminus R$ and $W_j\cap R\neq\varnothing$}\\
        \perp & \text{if $w\in W_j\setminus R$ and $W_j\cap R=\varnothing$}
    \end{cases}
\end{align*}

Per the above, if $m(w)\neq\perp$, then the stars centered at $w$ in $\mathcal I_W|_{\Sigma}$ are
\begin{align*}
    S(w, W\setminus\{w_1,\dots,w_{m(w)-1}\}),\, &S(w, W\setminus\{w_1,\dots,w_{m(w)-s-1}\}),\dots\\
    &\qquad\dots, S\lpr{w, W\setminus\lcr{w_1,\dots,w_{m(w)-\floor{\frac{m(w)-1}{s}}s-1}}}
\end{align*}
(where we use the convention that $\{w_1,\dots,w_0\}$ is the empty set).
In the case of $m(w)\neq\perp$, for every $0\leq t\leq \floor{\frac{m(w)-1}{s}}$, define the star 
\[Q'_t(w):=S(w, W\setminus\{w_1,\dots,w_{m(w)-ts-1}\}) \in \cI_W|_{\Sigma}.\]
Note that we have $Q'_0(w)\subseteq Q'_1(w)\subseteq \dots\subseteq Q'_{\floor{\frac{m(w)-1}{s}}}(w)$. Thus, we obtain that
\begin{align}\label{eq:I-sigma-1}
    \cI_W|_{\Sigma}=\left\{Q_t'(w) \lmid w\in \bigcup_{j=1}^sW_j,\, m(w)\neq\perp,\, 0\leq t\leq \floor{\frac{m(w)-1}{s}}\right\}.
\end{align}

We now investigate what the other sets in $\mathcal{I}|_{\Sigma}$ look like. We clearly have
\begin{align}\label{eq:I-sigma-2}
    \left\{\binom{V}{2}\right\}\bigl\vert_{\Sigma}=\{\Sigma\}.
\end{align}
Since every potential edge set in $\mathcal{J}_A\cup\mathcal{I}_U$ is disjoint from $\binom{W}{2}$, we get that
\begin{align}\label{eq:I-sigma-3}(\mathcal{J}_A)|_{\Sigma}\cup\mathcal{I}_U|_{\Sigma}=(\mathcal{J}_A)|_{S(A\cup U,W_\#)}\cup\mathcal{I}_U|_{S(A\cup U,W_\#)}=\{S(u,W_{\#})\mid u\in A\cup U\}.
\end{align}
Combining the above, we get that $\mathcal I|_{\Sigma}$ is the union of \cref{eq:I-sigma-1,eq:I-sigma-2,eq:I-sigma-3}.

We wish to show that $\mathcal I|_{\Sigma}$ is $\Theta(n^{1/2-\alpha})$-transformed. To do so, for every $Q_t'(w)\in \mathcal I_W|_{\Sigma}$ with $t\geq 1$, we define another set $Q_t(w):=Q'_t(w)\setminus Q'_{t-1}(w)$; for every $Q_0'(w)\in \mathcal I_W|_{\Sigma}$, we simply let $Q_0(w):=Q_0'(w)$. Then we define the new set family
\[\mathcal{Q}:= \left\{Q_t(w)\lmid w\in \bigcup_{j=1}^sW_j,\, m(w)\neq\perp,\, 0\leq t\leq \floor{\frac{m(w)-1}{s}}\right\}.\]

Note that $\mathcal I_W|_{\Sigma}$ is obtained from $\mathcal{Q}$ by applying an invertible linear transformation $T:\FF_2^{|\mathcal{Q}|}\to \FF_2^{|\mathcal{Q}|}$ (see \cref{def:linear-transform}). Extending $T$ with the identity map on $\{S(u,W_{\#})\mid u\in A\cup U\}\cup\{\Sigma\}$, we get that $\mathcal I|_{\Sigma}$ is obtained from $\mathcal{Q}\cup \{S(u,W_{\#})\mid u\in A\cup U\}\cup\{\Sigma\}$ by an invertible linear transformation $T':\FF_2^{|\mathcal{I}|}\to \FF_2^{|\mathcal{I}|}$.

It remains to show that we can order the sets in $\mathcal{Q}\cup \{S(u,W_{\#})\mid u\in A\cup U\}\cup\{\Sigma\}$ to form a $\Theta(n^{1/2-\alpha})$-layered sequence. We first note that $\binom{W_{\#}}{2}\subseteq \Sigma$ is disjoint from all the sets in $\mathcal{Q}\cup \{S(u,W_{\#})\mid u\in A\cup U\}$, with $\labs{\binom{W_{\#}}{2}}=\Omega(n^{1/2-\alpha})$. This means that we can put $\Sigma$ as the first set in the ordering.

Next, we note that the stars $S(u,W_{\#})$, $u\in A\cup U$ are pairwise disjoint and are disjoint from any set in $\mathcal{Q}$. Moreover, we know that $\abs{S(u,W_{\#})}=\abs{W_{\#}}=\Theta(n^{1/2-\alpha})$. Thus, we may put these stars right after $\Sigma$ in any order. It remains to show the following claim.

\begin{claim}\label{claim:QinWLayer}
        The stars in $\mathcal{Q}$ are $\Theta(n^{1/2-\alpha})$-layered.
    \end{claim}
    \begin{proof*}
    It  suffices to order the stars in $\mathcal{Q}$ such that for each star, there exists a subset of size at least $\Omega(n^{1/2 - \alpha})$ that does not appear in any other sets appearing later in the ordering. We call such a set a \emph{certificate} for the star. 
        
        Consider any ordering where all the  $Q_0(w)$ appear earlier than all the  $Q_t(w), t\geq 1$. For each $Q_0(w)$, we set its certificate as $S(w,W_{\#})$. For each $Q_t(w)$, we set its certificate as itself. It is clear that each certificate has size at least $\min\{s,|W_\#|\}=\Omega(n^{1/2-\alpha})$, and it is easy to check that each star contains its certificate.
        
        Note that for any $w$, the potential edges in $S(w,W_{\#})$ appear only in $Q_0(w)$. Thus $S(w,W_{\#})$ is indeed a certificate of $Q_0(w)$. Next, consider any potential edge $\{w,w'\}\in Q_t(w)$ for some fixed $t\geq 1$. Observe that this potential edge can only appear in $Q_t(w)$ or $Q_0(w')$. Thus, $Q_t(w)$ is indeed a certificate of $Q_t(w)$ since $Q_0(w')$ appears earlier than $Q_t(w)$ in the ordering.
    \end{proof*}
Thus, we are able to order the sets in $\mathcal{Q}\cup \{S(u,W_{\#})\mid u\in A\cup U\}\cup\{\Sigma\}$ to form a $\Theta(n^{1/2-\alpha})$-layered sequence.
\end{proof}

Now we prove \cref{prop:process-success-prob}.

\begin{proof}[Proof of \cref{prop:process-success-prob}]

Fix any $\ptrU$. Suppose we have successfully run $\ptrU-1$ rounds of the $(U,W)$-removal. We wish to show that the $\ptrU$-th round succeeds with high probability.

Recall that before the $\ptrU$-th round in the $(U,W)$-removal, we revealed $\parity(\mathcal{I}^{(\ptrU-1)}_A)$, $\parity(\mathcal{I}^{(\ptrU-1)}_U)$, $\parity(\mathcal{I}^{(\ptrU-1)}_W)$, $G[\mathcal{I}^{(\ptrU-1)}_e]$. In addition, $G[A]$, $\parity(\mathcal{J}_A)$, $\parity(V)$ are revealed due to $G$ being a partially revealed graph. Let 
\[\Sigma_1=\binom{V}{2}\setminus\left(\binom{V\setminus R^{(\ptrU-1)}}{2}\cup\Sigma\right),\]
where we recall that $\Sigma = \binom{W}{2}\cup S(A\cup U, W_\#)$. For convenience, we reveal $G[\Sigma_1]$ in the following analysis.

    Note that for any event $\cE$,  
    \[\PP[G\in\cE]=\sum_{H}\PP[G[\Sigma_1]=H]\cdot\PP[G\in\cE\mid G[\Sigma_1]=H].\] 
    Therefore, it is sufficient to show that the $\ptrU$-th round succeeds with high probability given that $G[\Sigma_1]=H$ for any feasible $H$, i.e., $G[\Sigma_1]$ is revealed.
    
    We first observe that we can simplify the revealed information due to dependencies between the revealed information.
        \begin{observation}\label{obs:IA}
         $\parity(\mathcal{I}^{(\ptrU-1)}_A)$ is determined by $G[\Sigma_1]$ and $\parity(\mathcal{J}_A)$.
    \end{observation}
    \begin{proof*}Recall that
    \[
    \mathcal{I}^{(\ptrU-1)}_A=\{S(\overline u_{\ptrU'},V\setminus R^{(\ptrU'-1)})\mid 1\leq \ptrU'\leq \min\{\ptrU-1,|A|\}\}.
    \]
    For every such $\ptrU'$, since $\overline u_{\ptrU'}\in A$, $\parity(\overline u_{\ptrU'},V\setminus R^{(\ptrU'-1)})$ is determined by $\parity(\overline u_{\ptrU'},V)\in\parity(\mathcal J_A)$ and the graph $G[\overline u_{\ptrU'}, R^{(\ptrU'-1)}]$, which is a subgraph of $G[\Sigma_1]$.
    \end{proof*}
    \begin{observation}
        The graph $G[\mathcal{I}_e^{(\ptrU-1)}]$ is determined by $G[\Sigma_1]$.
    \end{observation}
    \begin{proof*}
Recall that $\mathcal{I}_e^{(\ptrU-1)}$ is the collection  of sets  $S(\overline u_{\ptrU'},W_j\setminus R^{(\ptrU'-1)})$ for all $\ptrU'\leq \ptrU-1$ such that $\parity(\overline u_{\ptrU'}, V\setminus R^{(\ptrU'-1)})=1$, where $j\in [s]$ is as specified in round $\ptrU'$ of \cref{def:lrremoval}.
    Therefore, it suffices to show that every $S(\overline u_{\ptrU'},W_j\setminus R^{(\ptrU'-1)})\subseteq \Sigma_1$. This follows because
    \begin{itemize}
        \item $\overline u_{\ptrU'}\notin V\setminus R^{(\ptrU-1)}$, which implies $S(\overline u_{\ptrU'},W_j\setminus R^{(\ptrU'-1)})\cap \binom{V\setminus R^{(\ptrU-1)}}{2}=\varnothing$,
        \item $\overline u_{\ptrU'}\notin W$, which implies $S(\overline u_{\ptrU'},W_j\setminus R^{(\ptrU'-1)})\cap \binom{W}{2}=\varnothing$, and
        \item $W_j\cap W_{\#}=\varnothing$, which implies $S(\overline u_{\ptrU'},W_j\setminus R^{(\ptrU'-1)})\cap S(A\cup U,W_{\#})=\varnothing$.\qedhere
    \end{itemize}
    \end{proof*}
 
        From these two claims, we conclude that the collection of all the information we revealed is
    \[\Gamma := \lpr{\parity\lpr{\mathcal{I}^{(\ptrU-1)}_U},\parity\lpr{\mathcal{I}^{(\ptrU-1)}_W},\parity(\mathcal{J}_A),\parity(V),G[A],G[\Sigma_1]}.\]
    
    Recall that in the $\ptrU$-th round, we query the following:
    \begin{enumerate}
        \item the parity of $S_1:=S(\overline u_{\ptrU},V\setminus R^{(\ptrU-1)})$;
    \end{enumerate}
    and also the following when $\parity_G(S_1)=1$:
    \begin{enumerate}\setcounter{enumi}{1}
        \item the graph $S_2:=S(\overline u_\ptrU, W_j\setminus R^{(\ptrU-1)})$, with $j=\ptrW\pmod s$;
        \item the parities of $S_3(w):=S(w, V\setminus R^{(\ptrU-1)})$ for all $w\in W_j\setminus R^{(\ptrU-1)}$ where $j=\ptrW\pmod s$.
    \end{enumerate}
    
    Fix an assignment $\Gamma_0\in\supp(\Gamma)$. Assume that before the $\ptrU$-th round, the information we revealed is given by $\Gamma=\Gamma_0$. For convenience, in the discussion in this paragraph and the next paragraph, we will view $R^{(\ptrU-1)}$ as a sequence rather than a set, ordered according to the order in which its vertices are removed during the process. Note that although the sequence $R^{(\ptrU-1)}$ consists of random vertices, it is determined by $\Gamma$, and hence the sequence $R^{(\ptrU-1)}$ is determined when conditioning on $\Gamma=\Gamma_0$. Moreover, when conditioning on $\parity_G(S_1)=1$, the sets $S_2$ and $S_3(w),w\in W_j\setminus R^{(\ptrU-1)}$ are well-defined and determined. We want to estimate the probability of the output to each query as if the random vertex sets are deterministic, so we need to decouple the dependency between $\Gamma$ and $R^{(\ptrU-1)}$ and also the dependency between $\parity_G(S_1)$ and well-defining $S_2$ and $S_3(w)$ for all $w\in W_j\setminus R^{(\ptrU-1)}$. Thus, we sample an independent $p$-random graph $G'$ on $V$ and let $\Gamma'$ be the set of random variables obtained from $\Gamma$ by replacing $G$ with $G'$, i.e., 
    \[\Gamma' := \lpr{\parity_{G'}\lpr{\mathcal{I}^{(\ptrU-1)}_U},\parity_{G'}\lpr{\mathcal{I}^{(\ptrU-1)}_W},\parity_{G'}(\mathcal{J}_A),\parity_{G'}(V),G'[A],G'[\Sigma_1]}.\]

    Let $R^{(\ptrU-1)}_0$ be the outcome of $R^{(\ptrU-1)}$ given $\Gamma=\Gamma_0$, which is deterministic. Then we know that the distribution of $(G\mid \Gamma=\Gamma_0)$ and $(G\mid \Gamma=\Gamma_0,R^{(\ptrU-1)}=R^{(\ptrU-1)}_0)$ are identical. That is, they are $0$-close, or using the notation in \cref{def:eps-close}, we write
    \[(G\mid \Gamma=\Gamma_0)=(G\mid \Gamma=\Gamma_0,R^{(\ptrU-1)}=R^{(\ptrU-1)}_0).\]
    When conditioning on $R^{(\ptrU-1)}=R^{(\ptrU-1)}_0$, the vertices involved in the definition of $\Gamma'$ are fixed. Thus, we know that $(G,\Gamma)$ and $(G',\Gamma')$ are conditionally independent given $R^{(\ptrU-1)}=R^{(\ptrU-1)}_0$. Therefore, we have 
    \[(G'\mid \Gamma'=\Gamma_0,\Gamma=\Gamma_0)=(G'\mid \Gamma'=\Gamma_0,R^{(\ptrU-1)}=R^{(\ptrU-1)}_0).\]
    Since both $(G\mid \Gamma=\Gamma_0,R^{(\ptrU-1)}=R^{(\ptrU-1)}_0)$ and $(G'\mid \Gamma'=\Gamma_0,R^{(\ptrU-1)}=R^{(\ptrU-1)}_0)$ account for the random graphs that satisfy the assignment $\Gamma_0$ given $R^{(\ptrU-1)}=R^{(\ptrU-1)}_0$, it follows that
    \[(G\mid \Gamma=\Gamma_0)=(G'\mid \Gamma'=\Gamma_0,\Gamma=\Gamma_0).
    \]
    Similarly, we also have
    \[(G\mid \parity_G(S_1)=1,\Gamma=\Gamma_0)=(G'\mid \parity_{G'}(S_1)=1,\Gamma'=\Gamma_0,\parity_G(S_1)=1,\Gamma=\Gamma_0).
    \]
    Let $\cE_1$ denote the event that $\Gamma=\Gamma_0$ and $\cE_2$ denote the event that both $\Gamma=\Gamma_0$ and $\parity_G(S_1)=1$.
    Suggested by the discussion above, we can focus on analyzing the distributions of $(G'\mid \Gamma',\cE_1)$ and $(G'\mid  \parity_{G'}(S_1),\Gamma',\cE_2)$. For convenience, we set $\mathcal{I}^{(\ptrU-1)}=\mathcal{J}_A\cup \mathcal{I}_U^{(\ptrU-1)}\cup \mathcal{I}_W^{(\ptrU-1)}\cup \lcr{\binom{V}{2}}$, and hence $\Gamma'$ is the collection of $\parity_{G'}(\mathcal{I}^{(\ptrU-1)})$, $G'[A]$, and $G'[\Sigma_1]$.

    We first analyze the distribution of $\parity_{G'}(S_1)$. If $\ptrU\leq \abs{A}$, i.e., $\overline u_{\ptrU}\in A$, then it follows from a similar argument as \cref{obs:IA} that the parity $\parity_{G'}(S_1)$ is determined by $\Gamma'$. If $\ptrU> \abs{A}$, i.e., $\overline u_{\ptrU}\in U$, then we have the following claim.
    \begin{claim}\label{claim:parS1-close-to-uniform}
        If $\ptrU>\abs{A}$, then $(\parity_{G'}(S_1)\mid \Gamma'=\Gamma_0,\cE_1)$ is $e^{-\Theta(n^{1/2-\alpha})}$-uniform.
    \end{claim}
    \begin{proof*}
         In this proof, we always condition on $\cE_1$, and we will drop this conditioning from our notations for convenience.
         By \cref{lem:UWlayered}, we know that $\mathcal{I}|_{\Sigma}$ is $\Theta(n^{1/2-\alpha})$-transformed, where $\mathcal{I}=\mathcal{J}_A\cup\mathcal{I}_U\cup\mathcal{I}_W\cup\{\binom{V}{2}\}$ is given in \cref{lem:UWlayered}. Since $\Sigma$ is disjoint from $\binom{A}{2}\cup \Sigma_1$, we know that $(G'[\Sigma]\mid G'[A],G'[\Sigma_1])$ is $p$-random. Therefore, by \cref{cor:transformed} and the fact that the number of potential edge sets in $\mathcal{I}$ is polynomial in $n$, we can conclude that $(\parity_{G'}(\mathcal{I}|_{\Sigma})\mid G'[A],G'[\Sigma_1])$ is $e^{-\Theta(n^{1/2-\alpha})}$-uniform. So, we can apply \cref{lem:Y-epsUniform} on the graph $(G'\mid G'[A],G'[\Sigma_1])$ with $\vec{X}$ being $\binom{V}{2}\setminus \Sigma$ and $\vec{X}'$ being $\Sigma$ and get that
     \begin{enumerate}[label=(\alph*)]
         \item the parities $(\parity_{G'}(\mathcal{I})\mid G'[A],G'[\Sigma_1])$ are $e^{-\Theta(n^{1/2-\alpha})}$-uniform, and\label{claim:E1analysis_enu1}
         \item the graph $(G'\setminus G'[\Sigma]\mid G'[A],G'[\Sigma_1])$ is $e^{-\Theta(n^{1/2-\alpha})}$-affected by $\parity_{G'}(\mathcal{I})$.
     \end{enumerate} 
     In particular, since $\mathcal{I}^{(\ptrU-1)}\cup \{S_1\}$ is a subset of $\mathcal{I}$, we can conclude from \labelcref{claim:E1analysis_enu1} that the parities $(\parity_{G'}(\mathcal{I}^{(\ptrU-1)}),\parity_{G'}(S_1)\mid G'[A],G'[\Sigma_1])$ are $e^{-\Theta(n^{1/2-\alpha})}$-uniform.
     Therefore, 
     \[(\parity_{G'}(S_1)\mid \parity_{G'}(\mathcal{I}^{(\ptrU-1)}),G'[A],G'[\Sigma_1])=(\parity_{G'}(S_1)\mid \Gamma')\]
     is also $e^{-\Theta(n^{1/2-\alpha})}$-uniform.
    \end{proof*}
    
    From the previous discussions, we know that if we further condition on  $\parity_G(S_1)=1$, then $S_2$ and $S_3(w)$ for all $w\in W_j\setminus R^{(\ptrU-1)}$ are well-defined and determined. In this case, we show that the distribution of the induced graph on $S_2$ and the parities of $S_3(w)$ for $w\in W_j\setminus R^{\ptrU-1}$ are close to what one would expect for a partially revealed graph. 
    \begin{claim}\label{claim:G'-close-to-uniform}
        We have the following:
        \begin{enumerate}
            \item the parities $\lpr{\lpr{\parity_{G'}(S_3(w))}_{w\in W_j\setminus R^{(\ptrU-1)}}\mid \parity_{G'}(S_1), \Gamma',\cE_2}$ are $e^{-\Theta(n^{1/2-\alpha})}$-uniform, and
            \item the graph $\lpr{G'[S_2]\mid \parity_{G'}(S_1), \Gamma',\lpr{\parity_{G'}(S_3(w))}_{w\in W_j\setminus R^{(\ptrU-1)}},\cE_2}$ is $e^{-\Theta(n^{1/2-\alpha})}$-close to the graph $(G'[S_2]\mid G'[A],G'[\Sigma_1],\cE_2)$.
        \end{enumerate}
    \end{claim}
    \begin{proof*}
    In this proof, we always condition on $\cE_2$, and we will drop this conditioning from our notations for convenience. 
    Similar to the proof of \cref{claim:parS1-close-to-uniform}, we know that the following holds:
     \begin{enumerate}[label=(\alph*)]
         \item the parities $(\parity_{G'}(\mathcal{I})\mid G'[A],G'[\Sigma_1])$ are $e^{-\Theta(n^{1/2-\alpha})}$-uniform,  and\label{claim:E2analysis_enu1}
         \item the graph $(G'\setminus G'[\Sigma]\mid G'[A],G'[\Sigma_1])$ is $e^{-\Theta(n^{1/2-\alpha})}$-affected by $\parity_{G'}(\mathcal{I})$.\label{claim:E2analysis_enu2}
     \end{enumerate} 
    If $\ptrU\leq \abs{A}$, we know that $\parity_{G'}(S_1)$ is determined by $\Gamma'$, and $\mathcal{I}^{(\ptrU-1)}\cup \{S_3(w)\mid w\in W_j\setminus R^{(\ptrU-1)}\}$ is a subset of $\mathcal{I}$. If $\ptrU> \abs{A}$, we know that $\mathcal{I}^{(\ptrU-1)}\cup\{S_1\}\cup \{S_3(w)\mid w\in W_j\setminus R^{(\ptrU-1)}\}$ is a subset of $\mathcal{I}$. In either case, we can always conclude that
    \begin{align*}
        &\lpr{\lpr{\parity_{G'}(S_3(w))}_{w\in W_j\setminus R^{(\ptrU-1)}}\mid \parity_{G'}(S_1), \Gamma'}\\
        &=\lpr{\lpr{\parity_{G'}(S_3(w))}_{w\in W_j\setminus R^{(\ptrU-1)}}\mid \parity_{G'}(S_1), \parity_{G'}(\mathcal{I}^{(\ptrU-1)}),G'[A],G'[\Sigma_1]}
    \end{align*}
    is $e^{-\Theta(n^{1/2-\alpha})}$-uniform from \labelcref{claim:E2analysis_enu1}.

    Similarly, from \labelcref{claim:E2analysis_enu2}, we know that the graph $(G'\setminus G'[\Sigma]\mid G'[A],G'[\Sigma_1])$ is $e^{-\Theta(n^{1/2-\alpha})}$-close to the graph $\lpr{G'\setminus G'[\Sigma]\mid \parity_{G'}(S_1), \Gamma',\lpr{\parity_{G'}(S_3(w))}_{w\in W_j\setminus R^{(\ptrU-1)}}}$. Since $S_2$ is a subset of $\binom{V}{2}\setminus \Sigma$, the second part of the claim follows. 
    \end{proof*}
    Now, we can estimate the success probability of the $\ptrU$-th round using \cref{claim:parS1-close-to-uniform,claim:G'-close-to-uniform}.
    \begin{claim}\label{claim:UWconclusion}
        If $\ptrU>\abs{A}$, i.e., $\overline u_{\ptrU}\in U$, then $\parity_G (\overline u_\ptrU,V\setminus R^{(\ptrU-1)}) = 0$ (or $1$) with probability $\frac{1}{2}\pm e^{-\Theta(n^{1/2-\alpha})}$ when it is revealed at the $\ptrU$-th round of the $(U,W)$-removal. 
        
        For any $\ptrU=1,\dots,\abs{A}+\abs{U}$, if $\parity_G (\overline u_\ptrU,V\setminus R^{(\ptrU-1)}) = 1$, then there exists a desired vertex $w\in W_j\setminus R^{(\ptrU-1)}$ with probability $1-e^{-\Theta\lpr{\abs{W_j\setminus R^{(\ptrU-1)}}}}$ such that this round succeeds, where $j = \ptrW\pmod{s}$ is defined in \cref{def:lrremoval}.
    \end{claim}
    \begin{proof*}
    We first show the first part of the lemma. We know from \cref{claim:parS1-close-to-uniform} that if $\overline u_{\ptrU}\in U$, then
    \[(\parity_G(S_1)\mid \Gamma=\Gamma_0)=(\parity_{G'}(S_1)\mid \Gamma'=\Gamma_0,\cE_1)\]
    is $e^{-\Theta(n^{1/2-\alpha})}$-uniform. Therefore, in the $\ptrU$-th round, $\parity_G(S_1)=\parity_G (\overline u_\ptrU,V\setminus R^{(\ptrU-1)}) = 1$ (or $0$) with probability $1/2\pm e^{-\Theta(n^{1/2-\alpha})}$.
    
    Now, we show the second part of the statement. Fix $\ptrU$ and suppose $\cE_2$ is true, then we revealed $G[S_2]$ and $\parity_G(S_3(w))_{w\in W_j\setminus R^{(\ptrU-1)}}$. We know that by \cref{claim:G'-close-to-uniform},
    \begin{align*}
        & \lpr{\lpr{\parity_{G}(S_3(w))}_{w\in W_j\setminus R^{(\ptrU-1)}}\mid \parity_G(S_1)=1, \Gamma=\Gamma_0}\\
         &= \lpr{\lpr{\parity_{G'}(S_3(w))}_{w\in W_j\setminus R^{(\ptrU-1)}}\mid \parity_{G'}(S_1)=1, \Gamma'=\Gamma_0,\cE_2}
    \end{align*}
    is $e^{-\Theta(n^{1/2-\alpha})}$-uniform. Therefore, by the Chernoff bound, with probability at least $1-e^{-\Theta\lpr{\abs{W_j\setminus R^{(\ptrU-1)}}}}$, there are at least $\Theta\lpr{\abs{W_j\setminus R^{(\ptrU-1)}}}$ vertices $w\in W_j\setminus R^{(\ptrU-1)}$ such that $\parity_{G'}(S_3(w))=0$. We denote the set of all such $w$ by $\mho$. 

    From \cref{claim:G'-close-to-uniform}, we know that the graph $\lpr{G'[S_2]\mid \parity_{G'}(S_1), \Gamma',\lpr{\parity_{G'}(S_3(w))}_{w\in W_j\setminus R^{(\ptrU-1)}}}$ is $e^{-\Theta(n^{1/2-\alpha})}$-close to $(G'[S_2]\mid G'[A],G'[\Sigma_1])$. Since $S_2\cap\lpr{\binom{A}{2}\cup\Sigma_1}=\varnothing$, we know that the graph $(G'[S_2]\mid G'[A],G'[\Sigma_1])$ is $p$-random. Therefore, among all the $w\in \mho$, the probability that there exists one of them satisfying that $\overline u_{\ptrU} w$ is an edge in $G'$ is at least 
    \[\lpr{1-e^{-\Theta(n^{1/2-\alpha})}}\lpr{1-\lpr{1-p}^{|\mho|}}=1-e^{-\Theta\lpr{|\mho|}}\]
    since $\abs{\mho}=O(n^{1/2-\alpha})$.
    Similar as before, the same also holds for $G$ in place of $G'$. Therefore, we know that with probability at least $1-e^{-\Theta\lpr{\abs{W_j\setminus R^{(\ptrU-1)}}}}$ the set $\mho$ has size $\Theta\lpr{\abs{W_j\setminus R^{(\ptrU-1)}}}$, and we can find a desired $w\in\mho$ with probability at least $1-e^{-\Theta\lpr{|\mho|}}=1-e^{-\Theta\lpr{\abs{W_j\setminus R^{(\ptrU-1)}}}}$ such that this round succeeds.
    \end{proof*}

    Finally, we can estimate the probability that the $(U,W)$-removal succeeds. Define for fixed $\ptrU\le |A|+|U|$ the event $\cE_{\text{bad},\,\ptrU}$ as the event where the number of $\ptrU'\leq \ptrU$ such that $\parity_G (\overline u_{\ptrU'},V\setminus R^{(\ptrU'-1)})=1$ is not in $\left[\frac{\abs{\ptrU}}{2}-n^{3/4},\frac{\abs{\ptrU}}{2}+n^{3/4}\right]$. Let $\cE_{\text{bad}} := \bigcup_{\ptrU = 1}^{|A|+|U|}\cE_{\text{bad},\,\ptrU}$. For a fixed $\ptrU$, since the expected number of $\ptrU'\leq \ptrU$ such that $\parity_G (\overline u_{\ptrU'},V\setminus R^{(\ptrU'-1)})=1$ is $(\frac{1}{2}\pm e^{-\Theta(n^{1/2-\alpha})})\abs{\ptrU}=\frac{\abs{\ptrU}}{2}\pm o(n^{3/4})$, we have $\P[\cE_{\text{bad},\,\ptrU}]\le e^{-\Theta(n^{1/2})}<e^{-\Theta(n^{1/2-\alpha})}$ by the Chernoff bound. Taking a union bound over all $\ptrU$, we have 
    \[\P[\cE_{\text{bad}}]\le ne^{-\Theta(n^{1/2-\alpha})}=e^{-\Theta(n^{1/2-\alpha})}.\]

    Note that when $\cE_{\text{bad}}$ does not hold, we have 
    \begin{align*}
        \abs{W_j\setminus R^{(\ptrU-1)}}&\geq \abs{W_j}-\frac{1}{s}\abs{R^{(\ptrU-1)}\cap W}-1\geq \frac{1}{s}\abs{W}-\frac{1}{s}\abs{R^{(\ptrU-1)}\cap W}-2
    \end{align*}
    Furthermore, we have $\abs{R^{(\ptrU-1)}\cap W}\leq \frac{\abs{\ptrU}}{2}+n^{3/4}$ since $W$ is used only in the $\ptrU'$-th round of the removal procedure when $\parity_G (\overline u_{\ptrU'},V\setminus R^{(\ptrU'-1)})=1$. Thus we have $\abs{W}-\abs{R^{(\ptrU-1)}\cap W}\geq \abs{W}-\frac{\abs{\ptrU}}{2}-n^{3/4}=\Theta(n)$ and hence $\abs{W_j\setminus R^{(\ptrU-1)}}=\Theta(n^{1/2-\alpha})$.

    To conclude, conditioning on the event of $\cE_{\text{bad}}$ failing, the $(U,W)$-removal process succeeds with probability at least $1-(\abs{A}+\abs{U})e^{-\Theta(n^{1/2-\alpha})}=1-e^{-\Theta(n^{1/2-\alpha})}$. Since with probability $1-e^{-\Theta(n^{1/2-\alpha})}$ the event $\cE_{\text{bad}}$ fails, the process succeeds with probability at least $1-e^{-\Theta(n^{1/2-\alpha})}$. Moreover, when  $\cE_{\text{bad}}$ fails, we have  $\abs{R^{(\abs{A}+\abs{U})}\cap W}\leq (\abs{A}+\abs{U})/2+n^{3/4}=(1\pm o(1))\frac{n}{4}$ after the final round. Therefore, we have $\abs{V_W}=(1\pm o(1))\frac{n}{4}$ in this case and this completes the proof.
    \end{proof}

\section{Proof of \cref{thm:recurrence}}\label{sec:algorithm}

As mentioned in the introduction, we will run the $(U, W)$-removal twice with different inputs and obtain two disjoint output sets $V_W$ if both processes are successful. If either one of the remaining subgraphs is even-degenerate, then the input graph $G$ is also even-degenerate. We will get a probability boost in this way by showing that the two remaining subgraphs are close to being independent.

\subsection{Running the $(U,W)$-removal twice}\label{subsec:intro}
Given a graph $G$ with the set of revealed vertices $A\subseteq V$, 
we fix a partition of $V\setminus A$ into two balanced parts $B,C$. We fix orders of vertices in $A,B,C$ to form sequences $A=\{a_1,\dots,a_{\abs{A}}\}, B=\{b_1,\dots,b_{\abs{B}}\}, C=\{c_1,\dots,c_{\abs{C}}\}$. 

Pick $s=\Theta(n^{1/2+\alpha})$. We further partition both $B$ and $C$ into $s+1$ balanced sized blocks $B=B_\#\sqcup  B_1\sqcup \dots\sqcup B_s$ and $C=C_\#\sqcup C_1\sqcup \dots\sqcup C_s$, such that vertices in these blocks are \textit{decreasing} in vertex order. That is, we pick $r_1,\dots,r_{s+1}$ and $r'_1,\dots,r'_{s+1}$ such that
\begin{enumerate}
    \item $r_{s+1}=|B|$ and $r'_{s+1}=|C|$,
    \item $r_1,r_2-r_1,\dots, r_{s+1}-r_{s}$ pairwise differ by at most 1,
    \item $r_1',r'_2-r'_1,\dots, r'_{s+1}-r'_{s}$ pairwise differ by at most 1.
\end{enumerate}
We then set
\begin{align*}
&B_{\#}=\{b_{r_{s}+1},\dots, b_{r_{s+1}}\},  B_1=\{b_{r_{s-1}+1},\dots, b_{r_{s}}\}, \dots,  B_s=\{b_1,\dots,b_{r_1}\},\\
&C_{\#}=\{c_{r'_{s}+1},\dots, c_{r'_{s+1}}\},  C_1=\{c_{r'_{s-1}+1},\dots, c_{r'_{s}}\}, \dots, C_s=\{c_1,\dots,c_{r'_1}\}.
\end{align*}
Note that $B_{\#}, B_1, \dots, B_s, $ and $C_{\#}, C_1, \dots, C_s$ all have sizes $\Theta(n^{1/2-\alpha})$.

We run the $(U, W)$-removal (\cref{def:lrremoval}) twice on the graph $G$ as follows. First, we run the $(U, W)$-removal with inputs $U=\{b_1,\dots,b_{|B|}\}$ and  $W= C_{\#}\sqcup C_1\sqcup\dots\sqcup C_s$, and call this process the \emph{$(B,C)$-removal}. Next, we run the $(U, W)$-removal with inputs $U=\{c_1,\dots,c_{|C|}\}$,  $W= B_{\#}\sqcup B_1\sqcup\dots\sqcup B_s$, and call this process the \emph{$(C,B)$-removal}.

By \cref{prop:process-success-prob}, both $(B,C)$-removal and $(C,B)$-removal fail with probability at most $e^{-\Theta(n^{1/2-\alpha})}$. Suppose both processes succeed. Then we let $V_C\subseteq C$ and $R_{BC}$ denote the outputs of $(B,C)$-removal; let $V_B\subseteq B$ and $R_{CB}$ denote the outputs of $(C,B)$-removal. Furthermore, we define $i_B=\abs{B}-\abs{V_B}$ to be the number of vertices removed from $B$ in the $(C,B)$-removal, and define $i_C=\abs{C}-\abs{V_C}$ to be the number of vertices removed from $C$ in the $(B,C)$-removal. By \cref{prop:process-success-prob} again, with probability $1-e^{-\Theta(n^{1/2-\alpha})}$, we have $i_B,i_C\in [\frac{n}{4}-o(n),\frac{n}{4}+o(n)]$. Let $i_B'=i_B\pmod s$, $i_C'=i_C\pmod s$ (recall the convention that $s\pmod s=s$).
Now, we set $\eta=\Theta(n^{1/2-\alpha})$ small enough such that the sets 
\[D_B^P:=\cup_{i=1}^\eta B_i,\quad D_B^Q:=\cup_{i=i_B'-\eta+1}^{i_B'} B_i,\quad D_C^P:=\cup_{i=1}^\eta C_i,\quad D_C^Q:=\cup_{i=i_C'-\eta+1}^{i_C'} C_i\]
all have size $\le 0.01 n^{1-2\alpha}$. Set
\[
A_B=(B_{\#}\cup D_B^P\cup D_B^Q)\cap V_B,\qquad A_C=(C_{\#}\cup D_C^P\cup D_C^Q)\cap V_C.
\]
Note that $\abs{A_B}\leq 0.1n^{1-2\alpha}\leq \abs{V_B}^{1-2\alpha}$ and $\abs{A_C}\leq 0.1n^{1-2\alpha}\leq \abs{V_C}^{1-2\alpha}$ both hold whenever $\abs{V_B},\abs{V_C}\in[\frac{n}{4}-o(n),\frac{n}{4}+o(n)]$. 
We will show that after revealing all the information in both processes and also some more information, the graphs $G[V_B]$ and $G[V_C]$ are $e^{-\Theta(n^{1/2-\alpha})}$-close to two independent partially revealed $p$-graphs with revealed parts $A_B$ and $A_C$ respectively.
See \cref{fig:enter-label} for an illustration of $B_{\#},B_1,\dots,B_s,D_B^P, D_B^Q,A_B$.

\subsection{Revealed information after running both processes}\label{subsec:revealed-info}
In this subsection, we summarize all the information we revealed due to
\begin{enumerate}
    \item $G$ being a partially revealed graph,

    \item  $(B,C)$-removal,

    \item  $(C,B)$-removal.
\end{enumerate}

First, we recall that $G$ being a partially revealed graph means that $G[A],\parity(\cJ_A),\parity(V)$ are revealed. 

Consider any instance of a \textit{successful} $(U,W)$-removal (\cref{def:lrremoval}), so that the procedure ends at round $\ptrU_0=\abs{A}+\abs{U}$.  
Let
\[
[\ptrU_0]_{0}:=\{\ptrU\in[\ptrU_0]\mid \parity(\overline u_\ptrU, V\setminus R^{(\ptrU-1)})=0\}
\quad\text{and}\quad
[\ptrU_0]_{1}:=\{\ptrU\in[\ptrU_0]\mid \parity(\overline u_\ptrU, V\setminus R^{(\ptrU-1)})=1\}
\]
denote the sets of indices for which we fall into the first and second cases, respectively, in step (2) of the  $(U,W)$-removal (\cref{def:lrremoval}).  
By the end of the procedure, we have the following potential edge sets:
\begin{itemize}
    \item $\mathcal{I}_A^{(\ptrU_0)}\sqcup\mathcal{I}_U^{(\ptrU_0)}:=\{S(\overline u_\ptrU, V\setminus R^{(p-1)})\mid \ptrU\in \{1,\dots, \ptrU_0\}\}$, where
    $\mathcal{I}_A^{(\ptrU_0)}$ consists of those stars centered at $\overline u_p\in A$
    and $\mathcal{I}_U^{(\ptrU_0)}$ consists of those stars centered at $\overline u_\ptrU\in U$;
    \item $\mathcal I_W^{(\ptrU_0)}:=\bigcup_{\ptrU\in [\ptrU_0]_{1}}\{S(w, V\setminus R^{(\ptrU-1)})\mid w\in W_j\}$, where for every $\ptrU\in [\ptrU_0]_{1}$ we have
    $j=\ptrW\pmod s$, and $\ptrW-1$ is the number of vertices removed from $W$
    in the first $\ptrU-1$ rounds;
    \item $\mathcal I_e^{(\ptrU_0)}:=\{S(\overline u_\ptrU, W_j\setminus R^{(\ptrU-1)}) \mid \ptrU\in [\ptrU_0]_{1}\}$, where for every $\ptrU\in [\ptrU_0]_{1}$ we define $j$  as above.
\end{itemize}
Throughout the $(U,W)$-removal, we revealed
$\parity(\mathcal I_A^{(\ptrU_0)})$,
$\parity(\mathcal I_U^{(\ptrU_0)})$,
$\parity(\mathcal I_W^{(\ptrU_0)})$,
and $G[\mathcal I_e^{(\ptrU_0)}]$.

Since we run the $(U,W)$-removal twice, we need to show that after combining the knowledge of $\parity(\cI_A^{(\ptrU_0)}),\parity(\cI_U^{(\ptrU_0)}),\parity(\cI_W^{(\ptrU_0)})$, $G[\mathcal I_e^{(\ptrU_0)}]$ in \textit{both} processes, this information still has little influence on the randomness of $G[V_B]$ and $G[V_C]$. 

\begin{definition}\label{def:info-removal}
    
Let $\cI_A^{BC}$, $\mathcal I_B^{P}$, $\mathcal I_C^{Q}$, $\mathcal I_e^{BC}$ be the sets $\cI_A^{(\ptrU_0)}$, $\cI_U^{(\ptrU_0)}$, $\cI_W^{(\ptrU_0)}$, $\cI_e^{(\ptrU_0)}$ in the $(B,C)$-removal. Let $\cI_A^{CB}$, $\cI_C^P$, $\cI_B^Q$, $\cI_e^{CB}$ be the sets $\cI_A^{(\ptrU_0)}$, $\cI_U^{(\ptrU_0)}$, $\cI_W^{(\ptrU_0)}$, $\cI_e^{(\ptrU_0)}$ in the $(C,B)$-removal.
\end{definition}

We will further define some vertex sets so that the proof is easier to write. 

\begin{definition}\label{def:BandTB}
    Assume that $R_{CB}\cap B = \{\beta_1,\dots,\beta_{i_B}\}$ and $R_{BC}\cap C =\{\gamma_1,\dots,\gamma_{i_C}\}$, where $\beta_1,\dots,\beta_{i_B}$ and $\gamma_1,\dots,\gamma_{i_C}$ are in the same order as in $R_{CB}$ and $R_{BC}$ respectively. We set 
\[B^P:=\{\beta_1,\dots,\beta_\eta\}\subseteq D_B^P,\quad B^Q:=\{\beta_{i_B-\eta+1},\dots,\beta_{i_B}\}\subseteq D_B^Q,\]
\[C^P:=\{\gamma_1,\dots,\gamma_\eta\}\subseteq D_C^P,\quad C^Q:=\{\gamma_{i_C-\eta+1},\dots,\gamma_{i_C}\}\subseteq D_C^Q.\]
For convenience and clarity, we also define the sets 
\[T_B^P := (V_B\cup B^Q)\setminus (D_B^P\cup B_{\#}), \quad T_B^Q := V_B\setminus A_B=V_B\setminus(B_{\#}\cup D_B^P\cup  D_B^Q), \]
\[T_C^P:= (V_B\cup C^Q)\setminus (D_C^P\cup C_{\#}), \quad T_C^Q:= V_C\setminus A_C=V_C\setminus(C_{\#}\cup D_C^P\cup  D_C^Q).\]
We may assume without loss of generality that $n$ is large enough so that $i_B-\eta+1>2s,i_C-\eta+1>2s$, and hence $B^P\cap B^Q=\varnothing,C^P\cap C^Q=\varnothing$. 
See \cref{fig:enter-label} for an illustration of $B^P,B^Q,T_B^P,T_B^Q$.
\end{definition}

\begin{figure}
    \centering
    \scalebox{1}{
\begin{tikzpicture}[scale=0.6]
\definecolor{myred}{RGB}{228,26,28}
\definecolor{myblue}{RGB}{55,126,184}
\definecolor{mygreen}{RGB}{77,175,74}
\definecolor{myorange}{RGB}{255,127,0}
\definecolor{mypurple}{RGB}{152,78,163}
\definecolor{mybrown}{RGB}{166,86,40}
\definecolor{mypink}{RGB}{247,129,191}
\definecolor{mygray}{RGB}{153,153,153}
\definecolor{mycyan}{RGB}{0,191,196}
\definecolor{myyellow}{RGB}{255,255,51}
% Define a path with rounded corners
    \draw[rounded corners=8pt,blue,line width=2pt] 
        (0-1,4-1)--(0-1,28+1)--(18+1,28+1)--(18+1,4-1)--cycle;
        
    \foreach \j in {5,7,9,11,13,15,17,19,21,23,25,27}{
        \draw[black,line width=2pt] 
            (0-1,\j)--(18+1,\j);
    }

    \draw[rounded corners=8pt,black,line width=2pt,fill=red,fill opacity=0.2] 
        (0-1,32-1)--(0-1,32+1)--(18+1,32+1)--(18+1,32-1)--cycle;

    \draw[blue,line width=2pt] 
        (9,29)--(9,11)--(6.2,11)--(6.2,3);

    \draw[rounded corners=8pt,mypurple,line width=2pt] 
        (0-1.5,21)--(0-1.5,29)--(19.8,29)--(19.8,21)--cycle;
        
    \draw[rounded corners=8pt,mypurple,line width=2pt] 
        (0-1.5,11)--(0-1.5,19)--(19.8,19)--(19.8,11)--cycle;
        
    % \draw[rounded corners=8pt,red,line width=2pt] 
    %     (10-0.8,28+0.8)--(10-0.8,22-0.8)--(18+0.8,22-0.8)--(18+0.8,28+0.8)--cycle;
        
    % \draw[rounded corners=8pt,red,line width=2pt] 
    %     (10-0.8,18+0.8)--(10-0.8,12-0.8)--(18+0.8,12-0.8)--(18+0.8,18+0.8)--cycle;
    
    \draw[draw=none,line width=2pt,fill=red,fill opacity=0.2] 
        (9.2,28.8)--(9.2,21.2)--(18.8,21.2)--(18.8,28.8)--cycle;
        
    \draw[draw=none,line width=2pt,fill=red,fill opacity=0.2]  
         (9.2,18.8)--(9.2,11.2)--(18.8,11.2)--(18.8,18.8)--cycle;
    
    % \draw[rounded corners=8pt,orange,line width=2pt] 
    %     (-0.8,21.2)--(-0.8,28.8)--(1.2,28.8)--(1.2,21.2)--cycle;
        
    % \draw[rounded corners=8pt,orange,line width=2pt] 
    %     (8-1.6,11.2)--(8-1.6,18.8)--(8.8,18.8)--(8.8,11.2)--cycle;
    
    \draw[draw=none,line width=2pt,fill=gray,fill opacity=0.2] 
        (-0.8,21.2)--(-0.8,28.8)--(1.2,28.8)--(1.2,21.2)--cycle;
        
    \draw[draw=none,line width=2pt,fill=gray,fill opacity=0.2] 
        (8-1.6,11.2)--(8-1.6,18.8)--(8.8,18.8)--(8.8,11.2)--cycle;

    \draw[draw=none,line width=2pt,fill=mycyan,fill opacity=0.2] 
         (9.2,20.8)--(9.2,19.2)--(18.8,19.2)--(18.8,20.8)--cycle;
         
    \draw[draw=none,line width=2pt,fill=mycyan,fill opacity=0.2] 
         (6.4,10.8)--(6.4,3.2)--(18.8,3.2)--(18.8,10.8)--cycle;

    \draw[rounded corners=8pt,mybrown,line width=2pt] 
        (5.8,3-0.4)--(5.8,19.4)--(8.6,19.4)--(8.6,21.4)--(19.4,21.4)--(19.4,3-0.4)--cycle;

    \draw[rounded corners=8pt,red,line width=2pt] 
        (18+1,32)--(18+2.2,32)--(18+2.2,16)--(18.8,16);

    \draw[rounded corners=8pt,red,line width=2pt] 
        (18.8,24)--(18+3,24);
        
    \draw[rounded corners=8pt,mycyan,line width=2pt] 
        (18.8,20)--(18+2.6,20)--(18+2.6,6)--(18.8,6);
        
    \draw[rounded corners=8pt,mypurple,line width=2pt] 
        (19.8,15)--(18+3,15);
    \draw[rounded corners=8pt,mypurple,line width=2pt] 
        (19.8,25)--(18+3,25);
    \draw[rounded corners=8pt,mycyan,line width=2pt] 
        (20.6,13)--(18+3,13);
    \draw[rounded corners=8pt,mybrown,line width=2pt] 
        (19.4,10)--(18+3,10);
    \foreach \i in {5,13} {
    \foreach \j in {4,6,8,12,14,18,22,26,28} {
        \node at (\i,\j) {$\dots$};
    }
    }
    
        \node at (9,32) {$\dots$};

        \node[anchor=west] at (0,28) {$\beta_1$};
        \node[anchor=west] at (0,26) {$\beta_2$};
        \node[anchor=west] at (0,22) {$\beta_\eta$};
        
        \node[anchor=west] at (0,8) {$\beta_{s-2}$};
        \node[anchor=west] at (0,6) {$\beta_{s-1}$};
        \node[anchor=west] at (0,4) {$\beta_s$};
        
        \node[anchor=west] at (2,28) {$\beta_{s+1}$};
        \node[anchor=west] at (2,26) {$\beta_{s+2}$};
        \node[anchor=west] at (2,22) {$\beta_{s+\eta}$};
        
        \node[anchor=west] at (2,8) {$\beta_{2s-2}$};
        \node[anchor=west] at (2,6) {$\beta_{2s-1}$};
        \node[anchor=west] at (2,4) {$\beta_{2s}$};
        
        \node[anchor=west] at (8-1.7,18-0.4) {$\beta_{i_B-\eta+1}$};
        \node[anchor=west] at (8-1.6,14-0.4) {$\beta_{i_B-1}$};
        \node[anchor=west] at (8-1.4,12) {$\beta_{i_B}$};
        
        \node[anchor=west] at (-3,28) {$B_1$};
        \node[anchor=west] at (-3,26) {$B_2$};
        \node[anchor=west] at (-2.8,24) {$\vdots$};
        \node[anchor=west] at (-3,22) {$B_\eta$};

        \node[anchor=west] at (-3,8) {$B_{s-2}$};
        \node[anchor=west] at (-3,6) {$B_{s-1}$};
        \node[anchor=west] at (-3,4) {$B_s$};
        %\node[anchor=west] at (18.5,30) {\textcolor{red}{$A_B$}};
        \node[anchor=west] at (-3,32) {$B_{\#}$};

        \node at (3,2) {\textcolor{blue}{$F_0$}};
        \node at (13,2) {\textcolor{blue}{$F_1$}};

        \node at (4,16) {$\vdots$};
        \node at (4,20) {$\vdots$};
        \node at (4,24) {$\vdots$};
        \node at (4,10) {$\vdots$};
        
        \node at (14,16) {$\vdots$};
        \node at (14,20) {$\vdots$};
        \node at (14,24) {$\vdots$};
        \node at (14,10) {$\vdots$};

        \node[anchor=west] at (21,25) {\textcolor{mypurple}{$D_B^P$}};
        \node[anchor=west] at (21,15) {\textcolor{mypurple}{$D_B^Q$}};
        %\node[anchor=west] at (-3,2) {$B_{\#}$};
        
        \node[anchor=west] at (21,24) {\textcolor{red}{$A_B$}};
        \node[anchor=west] at (21,10) {\textcolor{mybrown}{$T_B^P$}};
        \node[anchor=west] at (21,13) {\textcolor{mycyan}{$T_B^Q$}};
        
        \node[anchor=west] at (-0.3,24) {\textcolor{black}{$B^P$}};
        \node[anchor=west] at (7,16) {\textcolor{black}{$B^Q$}};

    \draw[black,line width=1.5pt]
    (-4+0.5,13)--(-4+0.5,17);
    \draw[black,line width=1.5pt]
    (-4.3+0.5,16.7)--(-4+0.5,17)--(-3.7+0.5,16.7);
        
        \node[anchor=west] at (-3.8+0.5,15) {$b_i$};
        \node[anchor=west] at (-5+0.5,12.5) {$i$ small};
        \node[anchor=west] at (-5+0.5,17.5) {$i$ large};
    \foreach \i in {0,2,8,10,16,18} {
    \foreach \j in {4,6,8,12,14,18,22,26,28} {
        \draw [fill] (\i,\j) circle (1.6pt);
    }
    }

    \foreach \i in {0,2,15,17}{
        \draw [fill] (\i,32) circle (1.6pt);        
    }
\end{tikzpicture}}
    \caption{An illustration of the sets defined in \cref{subsec:intro,def:BandTB}.
    The vertices $b_i$ with larger index $i$ in the ordering $b_1,\dots,b_{\abs{B}}$ are higher in this picture. Note that it might be the case that $B^Q\cap D_B^P\neq\varnothing$, and in this case we do not include the intersection in $T_B^P$. The sets $F_0,F_1$ will be defined later in \cref{subsec:comb-lemma-2}.}
    \label{fig:enter-label}
\end{figure}

The following $\Gamma$ consists of all the information we shall reveal in our proof, which in particular contains $G[A],\parity(\cJ_A),\parity(V)$ and those in \cref{def:info-removal}.

\begin{definition}\label{def:gamma}
    Let $\Gamma$ denote the collection of information that consists of
    \[
G[A],\parity(\cJ_A),\parity(V),\parity(\cI_A^{BC}),\parity(\cI_A^{CB}),\parity(\cI_B^P), \parity(\cI_B^Q), \parity(\cI_C^P), \parity(\cI_C^Q), G[\cI_e^{BC}], G[\cI_e^{CB}]
    \]
    and the following information (1)--(5):
    \begin{enumerate}
    \item $G[A_B]$ and $G[A_C]$,
    \item $\parity(\mathcal{J}_{A_B})$ and $\parity(\mathcal{J}_{A_C})$, where $\mathcal{J}_{A_B}=\{S(a,V_B)\mid a\in A_B\}$ and $\mathcal{J}_{A_C}=\{S(a,V_C)\mid a\in A_C\}$,
    \item $\parity(V_B)$ and $\parity(V_C)$,
    \item $G\setminus(G[B]\cup G[C])$, and
    \item $G[\Sigma_B]$ and $G[\Sigma_C]$, where 
\begin{align}\label{eq:sigma-B}
    \Sigma_B:=\binom{B}{2}\setminus\left(\binom{V_B}{2}\cup S(B^P,T_B^P)\cup S(B^Q,T_B^Q)\right),
\end{align}
\begin{align}\label{eq:sigma-C}
    \Sigma_C:=\binom{C}{2}\setminus\left(\binom{V_C}{2}\cup S(C^P,T_C^P)\cup S(C^Q,T_C^Q)\right).
\end{align}
\end{enumerate}
\end{definition}
\begin{remark}\label{rmk:gamma}
    Note that (1)--(3) in $\Gamma$ are exactly the information we shall reveal so that  $V_B$ and $V_C$ become partially revealed graphs. After revealing (4), the graphs $G[B],G[C]$ become independent. The last part is revealed so that later proofs will be easier to write.
    
    In hindsight, $\Gamma$ is defined as such so that in \cref{subsec:comb-lemma}, we can argue that the potential edge sets $S(B^P,T_B^P)\cup S(B^Q,T_B^Q)$  provide enough randomness in $G[V_B]$, and the potential edges in $S(C^P,T_C^P)\cup S(C^Q,T_C^Q)$ provide enough randomness in $G[V_C]$, to ``hide'' the partial information revealed about $G[V_B]$ and $G[V_C]$ during the two removals. We will further simplify $\Gamma$ in the next subsection.
\end{remark}

Suppose we reveal $\Gamma$. 
From (1)(4)(5) in $\Gamma$, we know that the only potential edges that remain unrevealed are in
\begin{align}\label{eq:unrevealed-edges}
    \binom{B}{2}\setminus \lpr{\Sigma_B\cup \binom{A_B}{2}}=\lpr{\binom{V_B}{2}\setminus \binom{A_B}{2}}\sqcup S(B^P,T_B^P)\sqcup S(B^Q,T_B^Q)
\end{align}
and 
\begin{align}
    \binom{C}{2}\setminus \lpr{\Sigma_C\cup \binom{A_C}{2}}=\lpr{\binom{V_C}{2}\setminus \binom{A_C}{2}}\sqcup S(C^P,T_C^P)\sqcup S(C^Q,T_C^Q),
\end{align}
but the other partial information listed in \cref{def:gamma} may affect their randomness. In the rest of the section, we will show that conditioning on everything in $\Gamma$, $G[V_B]$ and $G[V_C]$ still have enough randomness to be close to two independent partially-revealed $p$-graphs, with revealed parts $A_B$ and $A_C$ respectively.

One observation is that although the sequences of vertices $R_{BC},R_{CB}$ are chosen from a random process, they are indeed deterministic when conditioning on $\Gamma$, provided that both $(B,C)$-removal and $(C,B)$-removal succeed. Since all the vertex sets we defined previously are determined by $R_{BC},R_{CB}$, they are also deterministic when conditioned on $\Gamma$. This is crucial to simplifying our analysis in \cref{subsec:comb-lemma}.

\subsection{Analysis of the remaining graph $G[V_B], G[V_C]$}\label{subsec:analysis}

Let $\cE$ be the event that both $(B,C)$-removal and $(C,B)$-removal are successful, and $\abs{V_B},\abs{V_C}\in [\frac{n}{4}-o(n),\frac{n}{4}+o(n)]$. 
Conditioning on any fixed assignment $\Gamma=\Gamma_0$ such that $\cE$ holds, we define $G_B,G_C$ to be two independent partially revealed $p$-graphs on $V_B,V_C$, with revealed parts $A_B,A_C$ that agree with $G[V_B],G[V_C]$ respectively. We will show that $G[V_B]\cup G[V_C]$ is $e^{-\Theta(n^{1/2-\alpha})}$-close to $G_B\cup G_C$ when conditioning on $\Gamma=\Gamma_0$. Formally, we have the following proposition.

\begin{proposition}\label{prop:close-prg}
    Let $G=(V,E)$ be a partially revealed $p$-graph. Assume the random variables $\Gamma,\Gamma_0,\cE,G[V_B],G[V_C],G_B,G_C$ are defined as above so that $\Gamma_0\in \supp(\Gamma)$ and $\cE$ holds when conditioning on $\Gamma=\Gamma_0$. Then the graph $(G[V_B]\cup G[V_C]\mid\Gamma=\Gamma_0)$ is $e^{-\Theta(n^{1/2-\alpha})}$-close to $(G_B\cup G_C\mid\Gamma=\Gamma_0)$.
\end{proposition}

We will prove \cref{prop:close-prg} in \cref{subsec:comb-lemma,subsec:comb-lemma-2}, basically by applying the lemmas in \cref{sec:prelim} a number of times. We now prove \cref{thm:recurrence} assuming \cref{prop:close-prg}.

\begin{proof}[Proof of \cref{thm:recurrence}]

Notice that by \cref{prop:process-success-prob} and a union bound, we know that the probability $\PP[\cE]\geq 1-e^{-\Theta(n^{1/2-\alpha})}$.
Furthermore, we have
\[\PP\left[G[V_B] \cup G[V_C]=H\mid \cE\right]=\EE_{\Gamma_0}\left[\PP[G[V_B] \cup G[V_C]=H\mid \Gamma=\Gamma_0]\right],\]
where the expectation is taken over all assignments $\Gamma_0$ such that $\cE$ is true. By \cref{prop:close-prg}, the graph $(G[V_B]\cup G[V_C]\mid \Gamma=\Gamma_0)$ is $e^{-\Theta(n^{1/2-\alpha})}$-close to $(G_B\cup G_C\mid \Gamma=\Gamma_0)$. 
This implies that for any assignment $H$ of potential edges in $\binom{V_B}{2}\cup \binom{V_C}{2}$, we have
\begin{align*}
    &1-e^{-\Theta(n^{1/2-\alpha})}\leq \frac{\PP[G[V_B] \cup G[V_C]=H\mid \Gamma=\Gamma_0]}{\PP[G_B\cup G_C =H\mid \Gamma=\Gamma_0]}\leq \lpr{1-e^{-\Theta(n^{1/2-\alpha})}}^{-1},
\end{align*}
whenever $H$ is in $\supp(G[V_B] \cup G[V_C]\mid \Gamma=\Gamma_0)=\supp(G_B\cup G_C\mid \Gamma=\Gamma_0)$. In particular, by an averaging argument, we have
\begin{align*}
    &1-e^{-\Theta(n^{1/2-\alpha})}\leq \frac{\PP[G[V_B] \text{ or }G[V_C] \text{ is even-degenerate}\mid \Gamma=\Gamma_0]}{\PP[G_B \text{ or }G_C \text{ is even-degenerate}\mid \Gamma=\Gamma_0]}\leq \lpr{1-e^{-\Theta(n^{1/2-\alpha})}}^{-1}.
\end{align*}

Now we can lower bound the probability that $G$ is even-degenerate as
    \begin{align*}
        &\PP[G\text{ is even-degenerate}]\\
        &\geq \PP[\cE]\cdot \PP[G[V_B] \text{ or }G[V_C] \text{ is even-degenerate}\mid \cE]\\
        &\geq\lpr{1-e^{-\Theta(n^{1/2-\alpha})}}\cdot \EE_{\Gamma_0}\left[\PP[G[V_B] \text{ or }G[V_C] \text{ is even-degenerate}\mid \Gamma=\Gamma_0]\right]\\
        &\geq \lpr{1-e^{-\Theta(n^{1/2-\alpha})}}\cdot\EE_{\Gamma_0}\left[\PP[G_B \text{ or }G_C \text{ is even-degenerate}\mid \Gamma=\Gamma_0]\right]
    \end{align*}
    Since $G_B$ and $G_C$ are independent conditioning on $\Gamma=\Gamma_0$, we get
    \begin{align*}
        &\EE_{\Gamma_0}\left[\PP[G_B \text{ or }G_C \text{ is even-degenerate}\mid \Gamma=\Gamma_0]\right]\\
       &=\EE_{\Gamma_0}\left[1-\PP[G_B \text{ is not even-degenerate}\mid \Gamma=\Gamma_0]\cdot\PP[G_C \text{ is not even-degenerate}\mid \Gamma=\Gamma_0]\right].
    \end{align*}
    Recall that {\rm(1)--(3)} in $\Gamma$ (\cref{def:gamma}) specify
exactly the revealed information for the partially revealed $p$-graphs on $V_B$
and $V_C$ with revealed parts $A_B$ and $A_C$ (see \cref{def:prg}), while all remaining components of $\Gamma$ concern edges outside
$\binom{V_B}{2}$ and $\binom{V_C}{2}$. Thus, by the definition of $f$, we obtain
    \[\PP[G_B \text{ is not even-degenerate}\mid \Gamma=\Gamma_0]\leq f(\abs{V_B}),\]\[\PP[G_C \text{ is not even-degenerate}\mid \Gamma=\Gamma_0]\leq f(\abs{V_C}),\]
    which gives the desired recursion
    \begin{equation}\label{eq:main-recursion}
        1-f(n)\geq \lpr{1-e^{-\Theta(n^{1/2-\alpha})}} \lpr{1-\max_{n'\in [n/4-o(n),n/4+o(n)]}f(n')^2}.\qedhere
    \end{equation}
\end{proof}

\subsection{A reduction of \cref{prop:close-prg}}\label{subsec:comb-lemma}
In the following two subsections we prove \cref{prop:close-prg}. Since the sequences of vertices $R_{BC},R_{CB}$ are determined by $\Gamma$, we need to decouple this dependency so that we can apply the lemmas in \cref{sec:prelim}. From now on, we fix an assignment $\Gamma=\Gamma_0$. Then $R_{BC},R_{CB}$, and all the vertex subsets defined in \cref{subsec:revealed-info} are now determined. 

Similar to what we did in the proof of \cref{prop:process-success-prob}, we independently sample a new $p$-graph $G'$ on the vertex set $V$. Let $\Gamma'$ be the set of random variables obtained from $\Gamma$ by replacing $G$ with $G'$. Note that the distribution of $(G\mid \Gamma=\Gamma_0)$ is the same as the distribution of $(G'\mid \Gamma'=\Gamma_0,\Gamma=\Gamma_0)$. The advantage of defining $G'$ is that we can consider $(G'\mid \Gamma',\Gamma=\Gamma_0)$ for different assignments of $\Gamma'$ with the vertex sets being deterministic rather than being a random variable depending on the probabilistic process we defined. Thus, it is sufficient to show that $G'[V_B]\cup G'[V_C]$, when conditioning on $G'[A_B],G'[A_C],\parity_{G'}(\mathcal{J}_{A_B}),\parity_{G'}(\mathcal{J}_{A_C}),\parity_{G'}(V_B),\parity_{G'}(V_C),\Gamma=\Gamma_0$, is $e^{-\Theta(n^{1/2-\alpha})}$-affected by $\Gamma'$. 

For simplicity, in the rest of the section, we will abuse notation by dropping the conditioning on $\Gamma=\Gamma_0$ from the notations, since we will always condition on $\Gamma = \Gamma_0$ from now on. We will also write $G$ and $\Gamma$ in place of $G'$ and $\Gamma'$, and write $\parity$ instead of $\parity_{G'}$.

As mentioned in \cref{rmk:gamma}, we first simplify $\Gamma$ by removing $G[A]$, $\parity(\mathcal{J}_A)$, $\parity(\cI_A^{BC})$, $\parity(\cI_A^{CB})$, $G[\mathcal{I}_e^{BC}]$, $G[\mathcal{I}_e^{CB}]$, as these pieces of information are determined by $G\setminus(G[B]\cup G[C])$. Furthermore, we can remove $\parity(V)$ because of the following observation.
\begin{observation}
    The parity $\parity(V)$ is determined by $G\setminus(G[B]\cup G[C])$, $\parity(V_B)$, $\parity(\cI_B^Q)$, $\parity(V_C)$, and $\parity(\cI_C^Q)$.
\end{observation}
\begin{proof}
    We first note that $\parity(V)$ is determined by $G\setminus(G[B]\cup G[C])$, $\parity(B)$ and $\parity(C)$ since $G[V]$ is the disjoint union of $G\setminus(G[B]\cup G[C])$, $G[B]$ and $G[C]$. 
    
    Now, we claim that $\parity(B)$ is determined by $G\setminus(G[B]\cup G[C])$, $\parity(V_B)$ and $\parity(\cI_B^Q)$. Note that all the potential edge sets in $\cI_B^Q$ are disjoint from $\binom{C}{2}$, and $\cI_B^Q|_{\binom{B}{2}}$ contains $S(\beta_m,B\setminus\{\beta_1,\dots,\beta_{m-1}\})$ for all $1\leq m\leq i_B$. Therefore, $\parity(\beta_m,B\setminus\{\beta_1,\dots,\beta_{m-1}\})$ is determined by $G\setminus(G[B]\cup G[C])$ and $\parity(\cI_B^Q)$. Note that 
    \[\bigcup_{m=1}^{i_B}S(\beta_m,B\setminus\{\beta_1,\dots,\beta_{m-1}\})=\binom{B}{2}\setminus\binom{V_B}{2},\]
    and hence $\parity(B)$ is determined by $G\setminus(G[B]\cup G[C])$, $\parity(V_B)$ and $\parity(\cI_B^Q)$. Similarly, we know that $\parity(C)$ is determined by $G\setminus(G[B]\cup G[C])$, $\parity(V_C)$ and $\parity(\cI_C^Q)$.
\end{proof}

Then we divide the remaining information except for $G\setminus(G[B]\cup G[C])$ in $\Gamma$ into two parts $\Gamma_B,\Gamma_C$ as follows. Let $\Gamma_B$ be the collection of information 
\begin{align*}
    \Gamma_B := \lpr{\parity(\cI_B^P|_{\binom{B}{2}}),\parity(\cI_B^Q|_{\binom{B}{2}}),G[A_B],\parity(\mathcal{J}_{A_B}),\parity(V_B),G[\Sigma_B]}.
\end{align*}
Similarly, let $\Gamma_C$ be the collection of information
\begin{align*}
    \Gamma_C := \lpr{\parity(\cI_C^P|_{\binom{C}{2}}),\parity(\cI_C^Q|_{\binom{C}{2}}),G[A_C],\parity(\mathcal{J}_{A_C}),\parity(V_C),G[\Sigma_C]}.
\end{align*}
It follows that conditioning on $\Gamma$ and conditioning on $(\Gamma_B,\Gamma_C,G\setminus(G[B]\cup G[C]))$ produce the same distribution of $G$.
Furthermore, $G[B]$, $G[C]$ are conditionally independent given $\Gamma$, i.e., $(G[V_B]\mid\Gamma)$ and $(G[V_C]\mid\Gamma)$ are independent. 

Thus, we may focus on showing that $(G[V_B]\mid\Gamma)$ is close to a partially revealed $p$-graph on $V_B$, as the proof for $G[V_C]$ will follow identically by symmetry. Observe that
\begin{align*}
    (G[V_B]\mid\Gamma)=(G[V_B] \mid\Gamma_B,\Gamma_C,G\setminus(G[B]\cup G[C]))=(G[V_B]\mid\Gamma_B).
\end{align*}
So, it suffices to prove the following key lemma.

\begin{lemma}\label{lem:B_Affected}
   The graph $(G[V_B]\mid \Gamma_B)$ is $e^{-\Theta(n^{1/2-\alpha})}$-close to the partially revealed $p$-graph 
   \[(G[V_B]\mid G[A_B],\parity(\mathcal{J}_{A_B}),\parity(V_B)).\]
\end{lemma}

Note that \cref{lem:B_Affected} yields \cref{prop:close-prg}; given $G$ as in \cref{prop:close-prg}, if we \textit{only} condition on some assignments of $ G[A_B]$, $\parity(\mathcal{J}_{A_B})$, $\parity(V_B)$ that agree with $\Gamma_0$, then the graph $G[V_B]$ is equally distributed as the partially revealed graph $G_B$ defined in \cref{subsec:analysis}, and the same holds for $G[V_C]$, $G_C$ by symmetry.
The next subsection is devoted to proving \cref{lem:B_Affected}.

\subsection{Proof of \cref{lem:B_Affected}}\label{subsec:comb-lemma-2}

 We first characterize and clean up the information  $\parity\lpr{\cI_B^P|_{\binom{B}{2}}}$ and $\parity\lpr{\cI_B^Q|_{\binom{B}{2}}}$ in $\Gamma_B$ via the following \cref{obs:stars-in-P} to \cref{obs:IBRtoQ} and \cref{claim:stars}. The stars in $\parity\lpr{\cI_B^P|_{\binom{B}{2}}}$ are easily characterized by the following observation.

\begin{observation}\label{obs:stars-in-P}
     All the stars in $\cI_B^P|_{\binom{B}{2}}$ are of the form 
\[P(b_i):= S(b_i,\{b_{i+1},\dots,b_{\abs{B}}\}).\]
\end{observation}

We now investigate what the stars in $\cI_B^Q|_{\binom{B}{2}}$ look like. Recall from \cref{subsec:revealed-info} that $\{\beta_1,\dots,\beta_{i_B}\}$ is the set of vertices removed from $B$ in the $(C,B)$-removal. Furthermore, these vertices all lie in $\bigcup_{i=1}^s B_i=B\setminus B_{\#}$. Let
\[
F_0=\{\beta_1,\dots,\beta_{i_B}\}=R_{CB}\cap B, \quad F_1=\bigcup_{i=1}^s B_i\setminus F_0=B\setminus (B_{\#}\cup F_0)
\]
denote the set of removed and remaining vertices in $B\setminus B_{\#}$. By the same argument as in the derivation of \cref{eq:I-sigma-1}, we have the following characterization of the stars in $\cI_B^Q|_{\binom{B}{2}}$ (again, we use the convention that $\{\beta_1,\dots,\beta_0\}=\varnothing$), which we split into two cases depending on $\beta\in F_0$ or $\beta\in F_1$.

\begin{observation}\label{obs:QlocationF}
    For every $\beta=\beta_m\in F_0$, the stars in $\cI_B^Q|_{\binom{B}{2}}$ centered at $\beta_m$ are given by $\{Q'_t(\beta_m):0\leq t\leq\floor{\frac{m-1}{s}}\}$, with
    \[
    Q'_t(\beta_m):=S(\beta_m,B\setminus\{\beta_1,\dots,\beta_{m-ts-1}\}).
    \]
\end{observation}

\begin{observation}\label{obs:QlocationT}
    For every $\beta\in F_1$ with $\beta\in B_j$, the stars in $\cI_B^Q|_{\binom{B}{2}}$ centered at $\beta$ are given by $\{Q'_t(\beta):0\leq t\leq\floor{\frac{m_j-1}{s}}\}$, with
    \[
    m_j=\max\{m:m\leq i_B,\, m\equiv j\pmod{s} \}
    \]
    and
    \[
    Q'_t(\beta):=S(\beta,B\setminus\{\beta_1,\dots,\beta_{m_j-ts-1}\}).
    \]
     By the assumption that $i_B-\eta+1>2s$, the set $\{m:m\leq i_B,\, m\equiv j\pmod{s} \}$ is always non-empty, and hence $m_j$ is well-defined.
\end{observation}
Inspired by the two observations above, we define $m(\beta):=m$ if $\beta=\beta_m\in F_0$ and $m(\beta):=m_j$ if $\beta\in F_1\cap B_j$.
We know that in both cases, we have 
\[Q'_0(\beta)\subseteq Q'_1(\beta)\subseteq\dots \subseteq Q'_{\floor{\frac{m(\beta)-1}{s}}}(\beta).\] 
As in \cref{sec:uw-removal}, for  every $Q_t'(\beta)\in \mathcal I_B^Q|_{\binom{B}{2}}$ with $t\geq 1$, we define $Q_t(\beta):=Q_t'(\beta)\setminus Q'_{t-1}(\beta)$; for every $Q_0'(\beta)\in \mathcal I_B^Q|_{\binom{B}{2}}$, we  let $Q_0(\beta):=Q_0'(\beta)$. Define
\[\mathcal{Q}:= \left\{Q_t(\beta):\beta\in B\setminus B_{\#},0\leq t\leq \floor{\frac{m(\beta)-1}{s}}\right\}.\]
By a similar argument as in \cref{lem:UWlayered}, $\parity(\mathcal Q)$ is obtained from $\parity\lpr{\mathcal I_B^Q|_{\binom{B}{2}}}$ via an invertible linear transformation. So by \cref{obs:LinearTrans}, we have the following.

\begin{observation}\label{obs:IBRtoQ}
    We have
    \begin{align*}
        (G[V_B]\mid\Gamma_B)&=(G[V_B]\mid \parity\lpr{\cI_B^P|_{\binom{B}{2}}},\parity\lpr{\cI_B^Q|_{\binom{B}{2}}}, G[A_B],\parity(\mathcal{J}_{A_B}),\parity(V_B),G[\Sigma_B])\notag\\
        &=(G[V_B]\mid \parity\lpr{\cI_B^P|_{\binom{B}{2}}},\parity(\mathcal{Q}), G[A_B],\parity(\mathcal{J}_{A_B}),\parity(V_B),G[\Sigma_B]).
    \end{align*}
\end{observation}

To simplify and categorize the set of unrevealed potential edges, we make the following two definitions by considering restrictions of the stars in $\cI_B^P|_{\binom{B}{2}}$ and $\cQ$ in $\binom{B}{2}\setminus \lpr{\Sigma_B \cup \binom{A_B}{2}}$.
\begin{definition}\label{def:sigma012}
    Recall that 
\begin{align*}
    & B^P=\{\beta_1,\dots,\beta_\eta\}\subseteq D_B^P, &&B^Q=\{\beta_{i_B-\eta+1},\dots,\beta_{i_B}\}\subseteq D_B^Q,\\
    &T_B^P = (V_B\cup B^Q)\setminus (D_B^P\cup B_{\#}),  &&T_B^Q = V_B\setminus A_B=V_B\setminus(B_{\#}\cup D_B^P\cup  D_B^Q).
\end{align*}
Set
\[\Sigma_0:=\binom{V_B}{2}\setminus \binom{A_B}{2}, \quad \Sigma_1:=S(B^P,T_B^P), \quad \Sigma_2:=S(B^Q,T_B^Q)\]
so that we have the partition
\[\binom{B}{2}\setminus \lpr{\Sigma_B \cup \binom{A_B}{2}}=\Sigma_0\sqcup \Sigma_1\sqcup \Sigma_2.\] 
\end{definition}

\begin{definition}\label{def:PQQQ}
    We define the following sets
\begin{align}
    \mathcal{P}&:=\{P(b)\cap(\Sigma_0\cup\Sigma_1\cup\Sigma_2)\mid b\in T_B^P\},\label{eq:def-P}\\
\mathcal{Q}_0&:=\{Q_0(\beta)\cap(\Sigma_0\cup\Sigma_1\cup\Sigma_2)\mid \beta\in T_B^Q\}=\{S(\beta,V_B\cup B^Q)\mid \beta\in T_B^Q\},\label{eq:def-Q0}\\
\mathcal{Q}_1&:=\{Q_0(\beta)\cap(\Sigma_0\cup\Sigma_1\cup\Sigma_2)\mid \beta\in B^P\}=\{S(\beta,T_B^P)\mid \beta\in B^P\},\label{eq:def-Q1}\\
\mathcal{Q}_2&:=\{Q_0(\beta)\cap(\Sigma_0\cup\Sigma_1\cup\Sigma_2)\mid \beta\in B^Q\}=\{S(\beta,T_B^Q)\mid \beta\in B^Q\}.\label{eq:def-Q2}
\end{align}
\end{definition}
\begin{remark}
    We check the equality in each of \cref{eq:def-Q0}--\cref{eq:def-Q2} here. Recall from the definition that $Q_0(\beta)=S(\beta,B\setminus\{\beta_1,\dots,\beta_{m(\beta)-1}\})$.

    We first check \cref{eq:def-Q0}. Note that if $\beta\in T_B^Q$, then $m(\beta)\leq i_B-\eta$ since $\beta\notin B_{i_B-\eta+1}\cup\dots\cup B_{i_B}$. Therefore, $B\setminus\{\beta_1,\dots,\beta_{m(\beta)-1}\}$ contains $V_B\cup B^Q$. Thus, we have $Q_0(\beta)\cap \Sigma_0=S(\beta,V_B)$ and $Q_0(\beta)\cap \Sigma_2=S(\beta,B^Q)$. Since $B^P\subseteq \{\beta_1,\dots,\beta_{m(\beta)-1}\}$, we know that neither $\beta$ nor any leaf of $Q_0(\beta)$ can be in $B^P$. Thus, $Q_0(\beta)\cap \Sigma_1=\varnothing$, and hence the equality in \cref{eq:def-Q0} follows. 

    Next, we show \cref{eq:def-Q1}. For $\beta\in B^P$, it is clear that $Q_0(\beta)\cap \Sigma_1=S(\beta,T_B^P)$. Note that $\beta\notin V_B\cup B^Q\cup T_B^Q$ since $B^P$ is disjoint from $V_B\cup B^Q\cup T_B^Q$. Thus, we have $Q_0(\beta)\cap (\Sigma_0\cup \Sigma_2)=\varnothing$.

    Finally, \cref{eq:def-Q2} follows similarly. For $\beta\in B^Q$, it is clear that $Q_0(\beta)\cap \Sigma_2=S(\beta,T_B^Q)$. Note that $\beta\notin V_B\cup B^P\cup T_B^P$ since $B^Q$ is disjoint from $V_B\cup B^P\cup T_B^P$. Thus, we have $Q_0(\beta)\cap (\Sigma_0\cup \Sigma_1)=\varnothing$.
\end{remark}

We will show that among the parities in $\parity(\cI_B^P|_{\binom{B}{2}})$ and $\parity(\mathcal{Q})$, only those in $\parity(\mathcal{P})$, $\parity(\mathcal{Q}_0)$, $\parity(\mathcal{Q}_1)$, $\parity(\mathcal{Q}_2)$ remain unrevealed due to revealing the subgraphs $G[\Sigma_B]$ and $G[A_B]$.

\begin{claim}\label{claim:stars}
We have
    \begin{align}\label{eq:claim-1}
    &(G[V_B]\mid \parity\lpr{\cI_B^P|_{\binom{B}{2}}},\parity(\mathcal{Q}), G[A_B],\parity(\mathcal{J}_{A_B}),\parity(V_B),G[\Sigma_B])\notag\\
    &=\lpr{G[V_B]\mid \parity(\mathcal{P}),\parity(\mathcal{Q}_0),\parity(\mathcal{Q}_1),\parity(\mathcal{Q}_2), G[A_B],\parity(\mathcal{J}_{A_B}),\parity(V_B),G[\Sigma_B]}.
\end{align}
\end{claim}

\begin{proof}
We start with three initial observations about which  potential edge sets are revealed/unrevealed by $G[A_B]$ and $G[\Sigma_B]$. Each observation can be checked by going through the definitions in \cref{subsec:intro} and \cref{subsec:revealed-info}.
\begin{itemize}
    \item[(a)] All the unrevealed potential edges (i.e., those in $\binom{B}{2}\setminus (\Sigma_B \cup \binom{A_B}{2})$) lie in $\binom{V_B\cup B^P\cup B^Q}{2}$.
    \item[(b)] All the potential edges in $\binom{D_B^P\cup B_{\#}}{2}$ are revealed.

     Since  $\binom{B}{2}\setminus\Bigl(\Sigma_B\cup \binom{A_B}{2}\Bigr)
=\Sigma_0\cup\Sigma_1\cup\Sigma_2$, any edge
$uv\in \binom{D_B^P\cup B_{\#}}{2}$ can be unrevealed only if it lies in
$\Sigma_0\cup\Sigma_1\cup\Sigma_2$. Since $(D_B^P\cup B_{\#})\cap V_B\subseteq A_B$, we know that if $uv\in\binom{V_B}{2}$ then $uv\in\binom{A_B}{2}$. Thus, we cannot have $uv\in \Sigma_0$. Also, we know that $T_B^P$ is disjoint from $D_B^P\cup B_{\#}$, so we cannot have $uv\in \Sigma_1$. Similarly, $T_B^Q$ is also disjoint from $D_B^P\cup B_{\#}$, so we cannot have $uv\in \Sigma_2$. Therefore, all the potential edges in $\binom{D_B^P\cup B_{\#}}{2}$ must be revealed.
    \item[(c)] All the potential edges in $\binom{B^Q\cup A_B}{2}$ are revealed.

    By contradiction, suppose $uv\in \binom{B^Q\cup A_B}{2}$ is unrevealed. Since $\binom{A_B}{2}$ is revealed, we cannot have both $u,v\in A_B$. Without loss of generality, suppose  $u\in B^Q$ (in particular, $u\notin V_B$ and $u\notin B^P$), so by \cref{eq:unrevealed-edges} we must have $u\in B^Q$, $v\in T_B^Q$. This is impossible because   $T_B^Q$ is disjoint from $B^Q\cup A_B$.
\end{itemize}

With these observations, we are able to identify the subcollections of $\parity\lpr{\cI_B^P|_{\binom{B}{2}}},\parity(\mathcal{Q})$ that are revealed by $G[\Sigma_B]$ and $G[A_B]$.
    \begin{itemize}
        \item[(i)] If $b\in B$ and $b\notin T_B^P$, then $P(b)$ is revealed. 
        
        By (a) above, $P(b)$ is revealed for all $b\notin V_B\cup B^P\cup B^Q$. If $b\in (V_B\cup B^P\cup B^Q)\setminus T_B^P$, then since $B^P\subseteq D_B^P$ and  $(V_B\cup B^Q)\setminus T_B^P=(V_B\cup B^Q)\cap (D_B^P\cup B_\#)$, we have $b\in (D_B^P\cup B_\#)$. Since $D_B^P\cup B_\#$ are the $\eta+1$ sets in the partition that contain those $b\in B$ with the largest indices in the ordering $b_1,\dots, b_{\abs{B}}$, we get that $P(b)\subseteq \binom{D_B^P\cup B_\#}{2}$. Thus, by (b) above, $P(b)$ is revealed.

        \item[(ii)] All stars of the form $Q_t(\beta)$, $t\geq 1$ are revealed.

        Consider any star $Q_t(\beta)$ with $t\geq 1$. We know from the definition $Q_t(\beta)=Q_t'(\beta)\setminus Q'_{t-1}(\beta)$ that all edges sets in $Q_t(\beta)$ are of the form $\beta u$ with  $u\notin V_B$. By contradiction, suppose one such $\beta u$ is not revealed. Since $u\notin V_B$, by  (a) above, we know that $u\in B^P\cup B^Q$. 
        
        Recall that $B^P=\{\beta_1,\dots,\beta_\eta\}$ are the first  $\eta$ vertices removed from $B$ during the $(C,B)$-removal, and $B^Q=\{\beta_{i_B-\eta+1},\dots,\beta_{i_B}\}$ are the last $\eta$  vertices removed from $B$ during the $(C,B)$-removal.
If $u\in B^P \subseteq D_B^P$, then we know from (b) above that $\beta\notin D_B^P\cup B_\#$, which means that $\beta\in\bigcup_{i=\eta+1}^sB_i$. However, for all $\beta\in\bigcup_{i=\eta+1}^sB_i$, the star $Q_t'(\beta)$ for any $t$ does not contain the edges $\beta \beta_1,\dots,\beta\beta_\eta$ (see \cref{obs:QlocationF}), so it is impossible that $u\in B^P$. This is a contradiction.
        
        If $u\in B^Q$, then we know from (c) above that $\beta\notin B^Q\cup A_B$. In particular, $\beta$ does not lie in $B^Q\cup (V_B\cap D_B^Q)$, which means that $Q_0'(\beta)$ always contains the star $S(\beta,B^Q)$. This implies that $Q_t(\beta)\subseteq (Q_t'(\beta)\setminus Q_0'(\beta))$ is disjoint from $S(\beta,B^Q)$, contradicting the assumption that $u\in B^Q$.

        \item[(iii)] If $\beta\in B\setminus(V_B\cup B^P\cup B^Q)$, then by (a), the star $Q_0(\beta)$ is revealed.

        \item[(iv)] If $\beta\in A_B$, then $\parity(Q_0(\beta))$ is  revealed.

        Recall that for all $\beta\in B$, the star $Q_0(\beta)=Q_0'(\beta)=S(\beta,B\setminus\{\beta_1,\dots,\beta_{m(\beta)}-1\})$ contains $S(\beta,V_B)$. Since $\parity(\beta,V_B)\in\parity(\mathcal{J}_{A_B})$, it suffices to show that $Q_0(\beta)\setminus S(\beta,V_B)$ is revealed. 

        Again, by contradiction, suppose $\beta u\in Q_0(\beta)\setminus S(\beta,V_B)$ is unrevealed. We already know that $u\notin V_B$; since $\beta\in A_B$, combining (a) and (c) above, we know that $u\in B^P$. However, for all $\beta\notin B^P$, the star $Q_0(\beta)=Q_0'(\beta)$ does not contain the edges $\beta\beta_1,\dots,\beta\beta_\eta$, so it is impossible that $u\in B^P$. This is a contradiction.
    \end{itemize}
    Combining {\rm(i)}--{\rm(iv)} above, we see that the stars in
$\cI_B^P|_{\binom{B}{2}}$ and $\mathcal{Q}$ 
other than
\begin{align*}
    \{P(b):b\in T_B^P\}\cup\{Q_0(\beta):\beta\in B^P\cup B^Q\cup (V_B\setminus A_B)\}
\end{align*}
have their parities already
determined by $G[A_B]$ and $G[\Sigma_B]$.
We also note that these stars are indeed in $\cI_B^P|_{\binom{B}{2}}\cup \mathcal{Q}$.
Thus, restricting these remaining stars to the unrevealed edge sets
$\Sigma_0\cup\Sigma_1\cup\Sigma_2$ yields exactly
$\mathcal P\cup \mathcal Q_0\cup \mathcal Q_1\cup \mathcal Q_2$. Thus, we can drop the information that is determined by $G[A_B],G[\Sigma_B]$ and get
\begin{align*}
    &(G[V_B]\mid \parity\lpr{\cI_B^P|_{\binom{B}{2}}},\parity(\mathcal{Q}), G[A_B],\parity(\mathcal{J}_{A_B}),\parity(V_B),G[\Sigma_B])\\
    &=\lpr{G[V_B]\mid \parity(\mathcal{P}),\parity(\mathcal{Q}_0),\parity(\mathcal{Q}_1),\parity(\mathcal{Q}_2), G[A_B],\parity(\mathcal{J}_{A_B}),\parity(V_B),G[\Sigma_B]}.
\end{align*}
as desired.
\end{proof}

The following claim shows that up to $e^{-\Theta(n^{1/2-\alpha})}$-closeness we can remove the influence of conditioning on $\parity(\cP)$ and $\parity(\cQ_1)$ from the right-hand side of \cref{eq:claim-1}.

\begin{claim}\label{claim:drop-P-Q1} 
\begin{multline*}
    (G[V_B]\mid \parity(\mathcal{P}),\parity(\mathcal{Q}_0),\parity(\mathcal{Q}_1),\parity(\mathcal{Q}_2), G[A_B],\parity(\mathcal{J}_{A_B}),\parity(V_B),G[\Sigma_B]) \\
    \text{is $e^{-\Theta(n^{1/2-\alpha})}$-close to}\\
    (G[V_B]\mid \parity(\mathcal{Q}_0),\parity(\mathcal{Q}_2), G[A_B],\parity(\mathcal{J}_{A_B}),\parity(V_B),G[\Sigma_B]).
\end{multline*}
\end{claim}

\begin{proof}
Recall that $\Sigma_1=S(B^P,T_B^P)$. We start by showing that the bipartite graph 
\[
(G[\Sigma_1]\mid \parity(\mathcal{Q}_0),\parity(\mathcal{Q}_2), G[A_B],\parity(\mathcal{J}_{A_B}),\parity(V_B),G[\Sigma_B])
\]
is $p$-random. To show this, we verify that the random graph $G[\Sigma_1]$ is independent from all the random variables $\parity(\mathcal{Q}_0)$, $\parity(\mathcal{Q}_2), G[A_B]$, $\parity(\mathcal{J}_{A_B})$, $\parity(V_B)$, $G[\Sigma_B]$.
\begin{itemize}
    \item Edges in $G[A_B]$, $\mathcal{J}_{A_B}$, $\binom{V_B}{2}$ all lie in $\binom{V_B}{2}$, which is disjoint from $\Sigma_1=S(B^P,T_B^P)$.
    \item We know from \cref{def:sigma012} that $\Sigma_1$ is disjoint from $\Sigma_B$.
    \item Consider any edge in some star in $\mathcal Q_0$, which is of the form $\beta u$ with $\beta\in T_B^Q$ and $u\in V_B\cup B^Q$ (see \cref{eq:def-Q0}). This implies that $\beta u\in\binom{V_B\cup B^Q}{2}$, which is disjoint from $\Sigma_1$ (as $V_B\cup B^Q$ and $B^P$ are disjoint).
    \item Consider any edge in some star in $\mathcal Q_2$, which is of the form $\beta u$ with $\beta\in B^Q$ and $u\in T_B^Q$ (see \cref{eq:def-Q2}). This also implies that $\beta u\in\binom{V_B\cup B^Q}{2}$, so we know that $\beta u$ does not lie in $\Sigma_1$.
\end{itemize}
Thus, we know that $(G[\Sigma_1]\mid \parity(\mathcal{Q}_0),\parity(\mathcal{Q}_2), G[A_B],\parity(\mathcal{J}_{A_B}),\parity(V_B),G[\Sigma_B])$ is a $p$-random graph. Note that $P(b)\cap \Sigma_1=S(b,B^P)$ if $b\in T_B^P$. This is because the index of $b$ is smaller than the index of any $b'\in B^P$ in the ordering $b_1,\dots,b_{\abs{B}}$. Thus, we have 
\[\mathcal{P}|_{\Sigma_1}=\{S(b,B^P)\mid b\in T_B^P\},\,\mathcal{Q}_1|_{\Sigma_1}=\{S(\beta,T_B^P)\mid \beta\in B^P\}.\] 
Therefore, we may apply \cref{lem:bipartite} and conclude that $\parity(\mathcal{P}|_{\Sigma_1})\cup \parity(\mathcal{Q}_1|_{\Sigma_1})$ is fix-parity $e^{-\Theta(n^{1/2-\alpha})}$-uniform.

Applying \cref{lem:evensumepsUniform} with the random variables $\vec{X}=\Sigma_0\cup\Sigma_2,\vec{X}'=\Sigma_1$, $\vec{Y}=\parity(\mathcal{P})\cup \parity(\mathcal{Q}_1)$, and $\vec{Z}=\parity(\mathcal{P}|_{\Sigma_1})\cup \parity(\cQ_1|_{\Sigma_1})$, we get that 
\begin{multline}
    (G[V_B]\mid \parity(\mathcal{P}),\parity(\mathcal{Q}_0),\parity(\mathcal{Q}_1),\parity(\mathcal{Q}_2), G[A_B],\parity(\mathcal{J}_{A_B}),\parity(V_B),G[\Sigma_B]) \\
    \text{is $e^{-\Theta(n^{1/2-\alpha})}$-close to} \tag{$*$}\\
    (G[V_B]\mid \sum_{S\in \mathcal{P}\cup \mathcal{Q}_1}\parity(S) \text{ (mod } 2), \parity(\mathcal{Q}_0),\parity(\mathcal{Q}_2), G[A_B],\parity(\mathcal{J}_{A_B}),\parity(V_B),G[\Sigma_B]).
\end{multline}

To deduce \cref{claim:drop-P-Q1}, we further remove the parity sum information $\sum_{S\in \mathcal{P}\cup \mathcal{Q}_1}\parity(S)\pmod 2$ via the following observation.

\begin{observation}\label{obs:parity-P-Q1}
    The parity $\sum_{S\in \mathcal{P}\cup \mathcal{Q}_1}\parity(S)\pmod{2}$ is the same as $\parity(\Sigma_0\cup \Sigma_2)$, which is determined by $\parity(\mathcal{Q}_2),\parity(V_B),G[A_B]$.
\end{observation}
\begin{proof*}
    Recall that $\mathcal{P}$ is obtained from $\mathcal{I}_B^P|_{\binom{B}{2}}$ by removing the stars whose parities are revealed, and removing the potential edges that are revealed from the remaining stars. Thus, 
    \[\sum_{S\in \mathcal{P}}\parity(S)\equiv \sum_{S\in \mathcal{I}_B^P|_{\binom{B}{2}}}\parity(S\cap \lpr{\Sigma_0\cup\Sigma_1\cup\Sigma_2})\pmod{2}.\]
    Note that the stars in $\mathcal{I}_B^P|_{\binom{B}{2}}$ are all disjoint, and their union is $\binom{B}{2}$. Thus, 
    \[\sum_{S\in \mathcal{P}}\parity(S)\equiv\parity(\Sigma_0\cup\Sigma_1\cup\Sigma_2)\pmod{2}.\]
    On the other hand, we know that
    \[\sum_{S\in \mathcal{Q}}\parity(S)\equiv\parity(\Sigma_1)\pmod{2},\]
    and hence
    \[\sum_{S\in \mathcal{P}\cup\mathcal{Q}}\parity(S)\equiv\parity(\Sigma_0\cup\Sigma_2)\pmod{2}.\]
    Note that 
    \[\parity(\Sigma_2)\equiv\sum_{S\in\mathcal{Q}_2}\parity(S)\pmod{2} ,\]
    \[\parity(\Sigma_0)\equiv\parity(V_B)-\parity(A_B)\pmod{2}.\]
    Hence, $\parity(\Sigma_0\cup \Sigma_2)\equiv\parity(\Sigma_0)+\parity(\Sigma_2)\pmod{2}$ is determined by $\parity(\mathcal{Q}_2),\parity(V_B),G[A_B]$.
\end{proof*}

Combining \cref{obs:parity-P-Q1} with $(*)$ gives the claim.
\end{proof}

Then, the following claim shows that we can further remove up to $e^{-\Theta(n^{1/2-\alpha})}$-closeness the influence of conditioning on $\parity(\cQ_0)$ and $\parity(\cQ_2)$. The argument is similar to the proof of \cref{claim:drop-P-Q1}.

\begin{claim}\label{claim:drop-Q0-Q2}
\begin{multline*}
(G[V_B]\mid \parity(\mathcal{Q}_0),\parity(\mathcal{Q}_2), G[A_B],\parity(\mathcal{J}_{A_B}),\parity(V_B),G[\Sigma_B])\\
\text{is $e^{-\Theta(n^{1/2-\alpha})}$-close to}\\
(G[V_B]\mid  G[A_B],\parity(\mathcal{J}_{A_B}),\parity(V_B),G[\Sigma_B]).
\end{multline*}
\end{claim}

\begin{proof}
Recall that $\Sigma_2=S(B^Q,T_B^Q)$.  Again, we start by showing that the bipartite graph 
\[
(G[\Sigma_2]\mid G[A_B], \parity(\mathcal{J}_{A_B}), \parity(V_B),G[\Sigma_B])
\]
is $p$-random. To show this, we verify that the random graph $G[\Sigma_2]$ is independent from all the random variables $G[A_B], \parity(\mathcal{J}_{A_B}), \parity(V_B),G[\Sigma_B]$.

\begin{itemize}
    \item Edges in $G[A_B]$, $\mathcal{J}_{A_B}$, $\binom{V_B}{2}$ all lie in $\binom{V_B}{2}$, which is disjoint from $\Sigma_2=S(B^Q,T_B^Q)$.
    \item We know from \cref{def:sigma012} that $\Sigma_2$ is disjoint from $\Sigma_B$.
\end{itemize}
Thus, we know that $(G[\Sigma_2]\mid G[A_B], \parity(\mathcal{J}_{A_B}), \parity(V_B),G[\Sigma_B])$ is a $p$-random graph. 
Note that
\[\mathcal{Q}_0|_{\Sigma_2}=\{S(\beta,B^Q)\mid \beta\in T_B^Q\},\,\mathcal{Q}_2|_{\Sigma_2}=\{S(\beta,T_B^Q)\mid \beta\in B^Q\}.\]
Therefore, we may apply \cref{lem:bipartite} and conclude that $\parity(\mathcal{Q}_0|_{\Sigma_2})\cup \parity(\mathcal{Q}_2|_{\Sigma_2})$ is fix-parity $e^{-\Theta(n^{1/2-\alpha})}$-uniform.

Applying \cref{lem:evensumepsUniform} with the random variables $\vec{X}=\Sigma_0,\vec{X}'=\Sigma_2$, $\vec{Y}=\parity(\mathcal{Q}_0)\cup \parity(\mathcal{Q}_2)$, and $\vec{Z}=\parity(\mathcal{Q}_0|_{\Sigma_2})\cup \parity(\cQ_2|_{\Sigma_2})$, we get that 
\begin{multline}
    (G[V_B]\mid \parity(\mathcal{Q}_0),\parity(\mathcal{Q}_2), G[A_B],\parity(\mathcal{J}_{A_B}),\parity(V_B),G[\Sigma_B]) \\
    \text{is $e^{-\Theta(n^{1/2-\alpha})}$-close to} \tag{$**$}\\
    (G[V_B]\mid \sum_{S\in \mathcal{Q}_0\cup \mathcal{Q}_2}\parity(S) \text{ (mod } 2), G[A_B],\parity(\mathcal{J}_{A_B}),\parity(V_B),G[\Sigma_B]).
\end{multline}

To deduce \cref{claim:drop-Q0-Q2}, we further remove the parity sum information $\sum_{S\in \mathcal{Q}_0\cup \mathcal{Q}_2}\parity(S)\pmod 2$ via the following observation.

\begin{observation}\label{obs:parity-Q0-Q2}
    The parity $\sum_{S\in \mathcal{Q}_0\cup \mathcal{Q}_2}\parity(S)\pmod{2}$ is the same as $\parity(A_B,T_B^Q)$, which is determined by $\parity(\mathcal{J}_{A_B}),G[A_B]$.
\end{observation}
\begin{proof*}
    It follows directly from the definitions that every potential edge in $\Sigma_2$ appears exactly once in $\mathcal{Q}_0$ and exactly once in $\mathcal{Q}_2$. Thus, potential edges in $\Sigma_2$ do not contribute to the parity $\sum_{S\in \mathcal{Q}_0\cup \mathcal{Q}_2}\parity(S)\pmod{2}$. Furthermore, since every potential edge in $\Sigma_1$ is incident to $B^P$ which is disjoint from  $V_B, B^Q, T_B^Q$, we know that potential edges in $\Sigma_1$ do not appear in $\mathcal{Q}_0\cup \mathcal{Q}_2$ and thus do not contribute to the parity sum either.

    It remains to study the contribution of $\Sigma_0=\binom{V_B}{2}\setminus\binom{A_B}{2}=\binom{T_B^Q}{2}\cup S(A_B,T_B^Q)$ (recall that $V_B=A_B\cup T_B^Q$).
    Every potential edge  in $\binom{T_B^Q}{2}$ does not appear in $\mathcal Q_2$ (as $B^Q$ is disjoint from $T_B^Q\subseteq V_B$) and appears twice in $\mathcal{Q}_0$ (as $T_B^Q\subseteq V_B$). Thus, potential edges  in $\binom{T_B^Q}{2}$ do not contribute to $\sum_{S\in \mathcal{Q}_0\cup \mathcal{Q}_2}\parity(S)\pmod{2}$ either. 

    Finally, every potential edge in $S(A_B,T_B^Q)$ does not appear in $\mathcal Q_2$ (as $B^Q$ is disjoint from $A_B,T_B^Q\subseteq V_B$) and appears exactly once in $\mathcal Q_0$. Thus, we have
    \[\sum_{S\in \mathcal{Q}_0\cup \mathcal{Q}_2}\parity(S)=\parity(A_B,T_B^Q)=\sum_{a\in A_B}\parity(a,T_B^Q)\pmod{2}.\]
    
    For every $a\in A_B$, $\parity(a,T_B^Q)=\parity(a,V_B)-\parity(a,A_B)$ is determined by $\parity(\mathcal{J}_{A_B}) $  and $G[A_B]$. The second part of the observation follows.
\end{proof*}
Combining \cref{obs:parity-Q0-Q2} with $(**)$ gives the claim.
\end{proof}

Finally we prove \cref{lem:B_Affected}, which readily follows from combining the previous claims.
\begin{proof}[Proof of \cref{lem:B_Affected}]
    Since $\binom{V_B}{2}\cap \Sigma_B = \varnothing$ (recall \cref{eq:sigma-B}) and $G[A_B]$, $\parity(\mathcal{J}_{A_B})$, $\parity(V_B)$ only depend on $G[V_B]$, we have
\begin{align*}
    (G[V_B]\mid G[A_B],\parity(\mathcal{J}_{A_B}),\parity(V_B))=(G[V_B]\mid G[A_B],\parity(\mathcal{J}_{A_B}),\parity(V_B),G[\Sigma_B]).
\end{align*}
Thus, to show \cref{lem:B_Affected}, it suffices to show that
\begin{align*}
    &(G[V_B]\mid \Gamma_B)
    \text{ is $e^{-\Theta(n^{1/2-\alpha})}$-close to } 
    (G[V_B]\mid G[A_B],\parity(\mathcal{J}_{A_B}),\parity(V_B),G[\Sigma_B]).
\end{align*}
By \cref{obs:IBRtoQ,claim:stars,claim:drop-P-Q1,claim:drop-Q0-Q2}, we know that
\begin{align*}
    (G[V_B]\mid\Gamma_B)
    &=\lpr{G[V_B]\mid \parity(\mathcal{P}),\parity(\mathcal{Q}_0),\parity(\mathcal{Q}_1),\parity(\mathcal{Q}_2), G[A_B],\parity(\mathcal{J}_{A_B}),\parity(V_B),G[\Sigma_B]}
\end{align*}
is $e^{-\Theta(n^{1/2-\alpha})}$-close to
\[(G[V_B]\mid \parity(\mathcal{Q}_0),\parity(\mathcal{Q}_2), G[A_B],\parity(\mathcal{J}_{A_B}),\parity(V_B),G[\Sigma_B]).\]
and is thus $e^{-\Theta(n^{1/2-\alpha})}$-close to 
\[(G[V_B]\mid  G[A_B],\parity(\mathcal{J}_{A_B}),\parity(V_B),G[\Sigma_B]).\qedhere\]
\end{proof}

\section{Upper bound on $f_\alpha(n)$}\label{sec:upperbound-f}

In this section we show \cref{thm:main}. Consider any $p\in(0,1)$ and $0<\alpha<1/2$.
Recall that $f(n):=f_\alpha(n)$ is the maximum of the probability that $G$ is not even-degenerate over all $\alpha$-partially revealed $p$-graphs $G$ with $n$ vertices.

Before proving the exponential upper bound on $f(n)$, we first need to prove that $f(n)$ is bounded away from $1$. This is a technical issue due to the fact that we do not know whether $f(n)$ is monotone decreasing in $n$. Thus, we first prove \cref{lem:intervalBdd,lemma:constBdd}.

\begin{lemma}\label{lem:intervalBdd}
        There exists a constant $N_1\in \NN$ depending only on $p$ and $\alpha$ such $f(n)<1$ holds for all $n\geq N_1$.
    \end{lemma}

    \begin{proof}
        Pick $N_1\in\NN$ such that $3n^{1-2\alpha}+4\leq n$ holds for all $n\geq N_1$.
Suppose $n\geq N_1$.
        Let $G=(V,E)$ be a partially revealed $p$-graph on $n$ vertices with revealed part $A$ with $|A|\leq n^{1-2\alpha}$. It is sufficient to show that, given any instance $H\subseteq \binom{A}{2}$, $s\in\{0,1\}$, and $(s_a)_{a\in A}\in\{0,1\}^{\abs{A}}$, we can construct a specific graph $G_0$ on $V$ that is even-degenerate and $G_0[A]=H$, $\abs{E(G_0)}\equiv s\pmod{2}$, and $\deg_{G_0}(a)\equiv s_a\pmod{2}$ for all $a\in A$. 
        
        By our choice of $N_1$, we have $3|A|+4\leq |V|$ so that there will be enough vertices in the following construction. We construct the graph $G_0$ on $V$ as follows:
        \begin{enumerate}
            \item Start with the prescribed graph $G_0[A]=H$ and $|V|-|A|$ isolated vertices in $V\setminus A$.
            \item Take any $2|A|+4$ vertices from $V\setminus A$ and label them as 
        \[
        \bigcup_{a\in A}\{x_a,y_a\}\cup \{w,t_1,t_2,b\}.
        \]
        \item For every $a\in A$, if $\deg_H(a)\equiv s_a\pmod{2}$, add the edges $ax_a$, $ay_a$; if $\deg_H(a)\not\equiv s_a\pmod{2}$, only add the edge $ax_a$.
        \item For every $a\in A$, add the edge $wx_a$. Add edges $wt_1,wt_2,t_1t_2$. 
        \item If the edge number of the current graph does not agree with the prescribed parity $s$, add the edge $bw$.
        \end{enumerate}
        Let $G_0$ be the graph obtained after (1)--(5). It is clear that $G_0$ agrees with the prescribed information $G_0[A]=H$, $\abs{E(G_0)}\equiv s\pmod{2}$, and $\deg_{G_0}(a)\equiv s_a\pmod{2}$ for all $a\in A$.

        To show that $G_0$ is even-degenerate, take any ordering $A=\{a_1,\dots,a_{|A|}\}$. We sequentially remove $a_1,\dots,a_{|A|}$ from $G_0$ as follows. For $i=1,\dots,|A|$, if the parity of $\deg(a_i)$ in the current graph is odd, then we first remove $x_{a_i}$ (as it has exactly two neighbors $a_i$, $w$) and then remove $a_i$; otherwise, we simply remove $a_i$.

        After this procedure, the remaining edges in the graph form a star centered at $w$ (with possible leaves $b$ and some $x_{a_i}$'s), and a triangle on $w,t_1,t_2$. Now if $w$ currently has even degree, we can remove $w$ and end at a single edge $t_1t_2$; if $w$ currently has odd degree, we can remove $t_1$ and then remove $w$, and end at an empty graph.
    \end{proof}
    
\begin{lemma}\label{lemma:constBdd}
    There exist constants $N_0\in \NN$ and $\delta\in (0,1)$ depending only on $p$ and $\alpha$, such that $f(n) <1- \delta$ for all $n > N_0$.
\end{lemma}

\begin{proof}
We know from \cref{eq:main-recursion} that there exist $K>0$, $N_2\in\NN$ and $0<\delta_2<1/4$ such that for all $n\geq N_2$, we have
    \begin{align}\label{eqn:recursion}
       1-f(n)\geq \lpr{1-e^{-Kn^{1/2-\alpha}}} \lpr{1-\max_{n'\in [n/4-\delta_2 n ,n/4+\delta_2 n]}f(n')}.
    \end{align}
Note that we have $f(n')$ instead of $f(n')^2$ in the last term, which follows from $f(n')\leq 1$. We set $N_0:=\max\{N_1,N_2\}$ (with $N_1$ as in \cref{lem:intervalBdd}) and pick $N\in\NN$ such that
\begin{itemize}
    \item $\prod_{n=N}^\infty (1-e^{-Kn^{1/2-\alpha}})>1/2$;
    \item $N_0<N/4-\delta_2 N$.
\end{itemize}
The existence of $N$ follows from the fact that $\sum_{n=1}^\infty e^{-Kn^{1/2-\alpha}}$ converges.
By \cref{lem:intervalBdd}, we have $\delta_1:=\min_{n\in [N_0,N]}(1-f(n))>0$.

Define $g:\{N,N+1,\dots\}\to\NN$ by $$g(n)=\text{argmax}_{[n/4-\delta_2 n,n/4+\delta_2 n]}f.$$
For all $n\geq N$, pick the minimum $k\in\NN$ such that $ g^{(k)}(n)< N$. Note that  we have $g^{(k)}\geq N_0$ as $g^{(k-1)}\geq N$.
Thus by \cref{eqn:recursion}, we have
\begin{align*}
    1-f(n)&\geq (1-e^{-Kn^{1/2-\alpha}})(1-f(g(n)))\\
    &\geq (1-e^{-Kn^{1/2-\alpha}})(1-e^{-Kg(n)^{1/2-\alpha}})(1-f(g^{(2)}(n)))\\
    &\geq \dots \\
    &\geq (1-f(g^{(k)}(n)))\prod_{i=0}^{k-1}(1-e^{-K(g^{(i)}(n))^{1/2-\alpha}})\\
    &\geq \delta_1\prod_{n=N}^\infty(1-e^{-Kn^{1/2-\alpha}})> \delta_1/2.
\end{align*}
Here the second last inequality follows from our choice of $N$ and the fact that $N_0\leq g^{(k)}<N$. Setting $\delta=\delta_1/2$ gives the result.
\end{proof}

We can now prove an exponential upper bound on $f(n)$, which immediately implies \cref{thm:main}.

\begin{theorem}
  $f(n) \le e^{-\Omega(n^{1/2-\alpha})}$.
\end{theorem}

\begin{proof}
With hindsight, we fix a real number $c$ such that $0<c<\frac{1}{4}-(\frac{1}{4})^{1/(1-2\alpha)}$. Note that since $1/(1-2\alpha)>1$, we have $\frac{1}{4}>(\frac{1}{4})^{1/(1-2\alpha)}$, and hence $c$ exists. Equivalently, we pick $c$ such that $\zeta:=2\cdot(\frac{1}{4}-c)^{\frac{1}{2}-\alpha}-1>0$.
From \cref{thm:recurrence}, we get that there exist $K>0$ and $N_3\in\NN$ such that for all $n\geq N_3$, we have
\begin{align}\label{eqn:recursion-2}
    f(n)\le e^{-Kn^{1/2-\alpha}} + \max_{n'\in [n/4-c n,n/4+c n]}f(n')^2.
\end{align}

We first show that $f(n)=o(1)$. 
\begin{claim}\label{claim:o1}
    For any $\eps>0$, there exists $M_\eps$ such that $f(n)<\eps$ for all $n>M_\eps$.
\end{claim}
\begin{proof*}
Let $x:=\limsup_{n\rightarrow\infty}f(n)$. From \cref{eqn:recursion-2}, we know that $x\leq 0+x^2$. By \cref{lemma:constBdd}, we know that $x<1$. Hence $x=0$.
\end{proof*}

We can now prove the theorem. With hindsight, pick $\veps>0$ such that
\[\veps+\veps^{\zeta}<1.\]
From \cref{claim:o1}, there exists $M_{\veps}$ such that $f(n)<\veps$ for all $n>M_{\veps}$. Now, we pick $M$ large enough such that 
\begin{itemize}
    \item $\ln(\frac{1}{\veps})/M^{1/2-\alpha}<K/2$;
    \item $(1/4-c)M>M_{\veps}$.
\end{itemize}
Let $K_0:=\ln(\frac{1}{\veps})/M^{1/2-\alpha}$, so we have $K_0<K/2$ and $e^{-K_0M^{1/2-\alpha}}=\veps$. We show by induction that 
\begin{align}\label{eq:exp-bound}
    f(n)\leq e^{-K_0n^{1/2-\alpha}}\qquad \text{for all $n\geq M_{\veps}$}.
\end{align}
For the base case, if $M_{\veps}\leq n\leq M$, then we have
\[f(n)< \veps= e^{-K_0M^{1/2-\alpha}}\leq e^{-K_0n^{1/2-\alpha}}.\]
Inductively, suppose $n>M$ and \cref{eq:exp-bound} holds for all $M_{\veps}\leq n'< n-1$. From \cref{eqn:recursion-2}, we have
\begin{align*}
    f(n)&\leq  e^{-Kn^{1/2-\alpha}} + \max_{n'\in [n/4-c n,n/4+c n]}f(n')^2\\
    &\leq e^{-Kn^{1/2-\alpha}}+e^{-2K_0(n/4-cn)^{1/2-\alpha}}\\
    &=e^{-K_0n^{1/2-\alpha}}\lpr{e^{-(K-K_0)n^{1/2-\alpha}}+e^{-\lpr{2(1/4-c)^{1/2-\alpha}-1}K_0n^{1/2-\alpha}}}\\
    &=e^{-K_0n^{1/2-\alpha}}\lpr{e^{-(K-K_0)n^{1/2-\alpha}}+e^{-\zeta K_0n^{1/2-\alpha}}}\\
    &\leq e^{-K_0n^{1/2-\alpha}}\lpr{e^{-K_0M^{1/2-\alpha}}+e^{-\zeta K_0M^{1/2-\alpha}}}\\
    &=e^{-K_0n^{1/2-\alpha}}\lpr{\veps+\veps^{\zeta}}<e^{-K_0n^{1/2-\alpha}}.
\end{align*}
Here, the second inequality follows from the inductive hypothesis and the last inequality follows from $K_0<K/2$. The inductive hypothesis applies because we have $n'>n/4-cn\geq (1/4-c)M>M_{\veps}$. 
This finishes the proof.
\end{proof}

\section*{Acknowledgment}
We would like to thank Lisa Sauermann and Hung-Hsun Hans Yu for helpful discussions. We would also like to thank anonymous referees for carefully reading our manuscript and providing detailed suggestions on improving the paper.

\bibliographystyle{plain}
\bibliography{bib.bib}
\end{document}